\documentclass[10pt,a4paper,reqno]{amsart}

\usepackage{lmodern}
\usepackage[T1]{fontenc}
\usepackage[utf8]{inputenc}
\usepackage[english]{babel}
\usepackage{microtype} 
\usepackage[yyyymmdd,hhmmss]{datetime}
\usepackage{float}

\usepackage{slashed,url,bm,amssymb,mathrsfs,mathtools,todonotes,xparse,booktabs,graphicx}
\usepackage[centertableaux]{ytableau}
\makeatletter
\edef\boxframe@normal@YT{\boxframe@YT}
\usepackage[pdftex]{hyperref}
\usepackage{float,placeins}

\usepackage{blindtext} 
\usepackage{epigraph} 

\usepackage{array}
\usepackage[most]{tcolorbox}

\usepackage{amsmath,amssymb,amsthm}
\usepackage{graphicx}
\usepackage[margin=1in]{geometry}
\usepackage{xcolor}
\usepackage{tikz}
\usepackage{mathrsfs}
\usetikzlibrary{decorations.pathreplacing}
\usepackage{mathscinet}

\newtheorem{theorem}{Theorem}
\newtheorem{lemma}{Lemma}
\newtheorem{proposition}{Proposition}

\newtheorem{corollary}{Corollary}
\theoremstyle{definition}
\newtheorem{definition}{Definition}

\newtheorem{example}{Example}

\newcommand{\al}{\alpha}
\newcommand{\J}{J}

\newcommand{\Js}{J^{\#}}
\newcommand{\arm}{\operatorname{arm}}
\newcommand{\leg}{\operatorname{leg}}
\newcommand{\QQ}{\mathbb{Q}}
\newcommand{\ZZ}{\mathbb{Z}}
\newcommand{\Lam}{\Lambda}
\newcommand{\sh}{\operatorname{sh}}
\newcommand{\ellp}{\ell}
\newcommand{\normJ}{{j}}
\newcommand{\Lamf}{\Lambda_{(q,t)}}
\newcommand{\Laml}{\Lambda_{[q^{\pm},t^{\pm}]}}
\newcommand{\Lamsf}{\Lambda^*_{(q,t)}}
\newcommand{\Lamsl}{\Lambda^*_{[q^{\pm},t^{\pm}]}}
\newcommand{\Lamfa}{\Lam_{(\al)}}
\newcommand{\LamP}{\Lam_{[\al]}}
\newcommand{\LamA}{\Lam_{[\al^{\pm}]}}
\newcommand{\LamsA}{\Lam^*_{[\al^{\pm}]}}
\newcommand{\Lamsfa}{\Lam^*_{(\al)}}
\newcommand{\BC}{\mathbb{C}}

\newcommand{\Hilb}{\operatorname{Hilb}}

\newcommand{\HT}{H^{*}_{T}}
\newcommand{\KT}{K_{T}}

\definecolor{boxU}{RGB}{80,80,255}
\definecolor{boxD}{RGB}{255,80,80}

\title{Congruences of shifted Jack Littlewood--Richardson coefficients}
\author[R. Mickler]{Ryan Mickler}
\address{Singulariti Research, 55 University Street, Carlton, 3053, Australia}
\email{ry.mickler@gmail.com}

\begin{document}

\begin{abstract}
The shifted Jack Littlewood--Richardson coefficients generalize the
ordinary Jack coefficients and are Laurent polynomials in the Jack parameter
$\al$.  We prove a previously conjectured congruence: coefficients indexed by
triples differing by a single box move are congruent modulo the shared
$\al$-hook at the pivot.  We also prove a shifted Macdonald analogue, with a
power-of-$t$ twist, and establish that the normalized shifted Macdonald
coefficients are Laurent polynomials in $q$ and $t$.  The proofs combine
coincidences of shifted coordinates with Laurent-preserving shift transforms,
and the Macdonald input uses Knop's inversion formula and the integrality of the
Bergeron--Garsia--Haiman--Tesler operators.  Finally, we realize the Jack and
Macdonald congruences as necessary edge conditions on Hilbert schemes of
points, in equivariant cohomology and equivariant $K$-theory, respectively.
In the Macdonald case a tautological determinant twist accounts for the
power-of-$t$ normalization.

\end{abstract}
\maketitle

\subsection*{Introduction}

Consider two partitions $\lambda,\tilde\lambda$ of the same size which are
related by a single box move, i.e. moving one removable box from a row $a$ of
$\lambda$ and adding it to a row $b>a$.  We write
$\lambda\sim_p\tilde\lambda$ for this move, where the \emph{pivot} box is
$p:=(a,\lambda_b+1)$, i.e. the box which is in the row of the removed box and the
column of the added box (see Figure~\ref{fig:possible-pivots}). The two partitions
have a shared $\alpha$-hook length, $h$, obtained by evaluating the lower hook of
$\lambda$ and the upper hook of $\tilde\lambda$ (defined in
Section~\ref{sec:congruence}) both at $p$:
\begin{equation}\label{eq:shared-hook}
h:=h_\lambda(p)=h'_{\tilde\lambda}(p)
=\al(\lambda_a-\lambda_b-1)+(b-a).
\end{equation}
The reverse move exchanges the roles of $\lambda$ and $\tilde\lambda$, and the
congruences below are symmetric, so it suffices to consider moves from a row $a$ to
a strictly lower row $b>a$.  Then $b-a\ge1$, and $\lambda_a\ge\lambda_b+2$, since
$\tilde\lambda_a=\lambda_a-1$ and $\tilde\lambda_b=\lambda_b+1$ must satisfy
$\tilde\lambda_a\ge\tilde\lambda_b$.  Thus $h=\al x+y$ with $x,y\ge1$, so in
particular the shared hook is never proportional to $\al$.

\begin{figure}[t]
\centering
\begin{ytableau}
*(boxU) & & *(boxD) & *(boxU) & *(boxD) & *(boxU) & & & *(boxD) & *(boxU) & \\
        & & *(boxD) &         & *(boxD) &         & & & *(boxD)\\
*(boxU) & &         & *(boxU) &         & *(boxU) & & & \\
        & & *(boxD) &         & *(boxD) \\
*(boxU) & &         & *(boxU) & \\
*(boxU) & &
\end{ytableau}
\caption{Pivot boxes for $\lambda=(11,9,9,5,5,3)$.
Blue boxes lie in the row of a removable box and the column of an addable
box.  They give downward moves out of $\lambda$, with shared hook
$h_\lambda(p)$.  Red boxes lie in the column of a removable box and the row
of an addable box.  They give downward moves into $\lambda$, with shared
hook $h'_\lambda(p)$.}
\label{fig:possible-pivots}
\end{figure}

For triples of partitions $\sigma=(\mu,\nu,\lambda)$ and
$\tilde\sigma=(\tilde\mu,\tilde\nu,\tilde\lambda)$ we write
$\sigma\sim_p\tilde\sigma$ if exactly one of the three partitions changes, by a
single-box move with pivot $p$, and we call $p$ a \emph{pivot} of the triple.

The Jack pivot congruences were conjectured in
\cite[Conjectures~6.4 and~6.6]{Mickler:2026aa}.  Our first result proves them
directly, without the auxiliary constructions of that paper.

\begin{theorem}[Jack pivot congruences]\label{thm:top-pivots}
Let $(\mu,\nu,\lambda)\sim_p(\tilde\mu,\tilde\nu,\tilde\lambda)$ be a pivot
of triples with shared hook $h$.  The normalized shifted Jack
Littlewood--Richardson coefficients satisfy
\[
g^\lambda_{\mu\nu}\equiv g^{\tilde\lambda}_{\tilde\mu\tilde\nu}\pmod h
\qquad\text{in }\QQ[\al^{\pm1}].
\]

\end{theorem}
For $|\lambda|=|\mu|+|\nu|$, these are Stanley's ordinary Jack coefficients
$g^\lambda_{\mu\nu}=\langle J_\mu J_\nu,J_\lambda\rangle_\al$
\cite{Stanley:1989}, and more generally $g$ has the normalization of \cite{Alexandersson:2024aa}, recalled in
Section~\ref{sec:congruence}.\footnote{In \cite{Mickler:2026aa} this coefficient
is denoted $g_{\mu\nu;\lambda}$.}

The mechanism is elementary.  In the scaled shifted coordinates
$z_i(\lambda)=\al\lambda_i-i$, a pivot transposes the two affected
coordinates and changes them by $\pm h$
(Lemma~\ref{lem:shifted-coordinate-pivot}).  Hence every Laurent
$\al$-shifted symmetric function satisfies
$F(\lambda)\equiv F(\tilde\lambda)\pmod h$
(Lemma~\ref{lem:regular-pivot}).  We realize $g^\lambda_{\mu\nu}$ as a Laurent
unit times such an evaluation, using Lassalle's shift transform $\sh$
\cite{Lassalle:2008aa}.  The essential step is that both $\sh$ and its inverse
preserve Laurent coefficients (Corollary~\ref{cor:jack-laurent-iso}, and in
finitely many variables see Knop--Sahi \cite[Section~6]{Knop:1996}).
For instance, in top degree,
$g^\lambda_{\mu\nu}=\sh(\al^{|\lambda|}J_\mu J_\nu)(\lambda)$.

\noindent\begin{minipage}{\linewidth}
\begin{example}\label{ex:example1}
The partitions $\lambda=(4,3,3,1)$ and $\tilde\lambda=(4,3,2,2)$ have pivot
$p=(3,2)$ and shared hook $h=1+\al$.  For $\mu=(3,2,1)$ and $\nu=(2,2,2)$,
the coefficients are
\begin{align*}
g^{4331}_{321,222}
&=48\al^4(2+\al)^2(3+\al)^2(1+2\al)^2(1+3\al)(2+3\al)^2
  (24+171\al+284\al^2+116\al^3),\\
g^{4322}_{321,222}
&=288\al^5(2+\al)^3(3+\al)^2(1+2\al)^2(2+3\al)(3+4\al)^2
  (2+11\al+2\al^2).
\end{align*}
Their difference factors as
\[
g^\lambda_{\mu\nu}-g^{\tilde\lambda}_{\mu\nu}
=-48\al^4\underbrace{(1+\al)}_{h}(2+\al)^2(3+\al)^2(1+2\al)^2(2+3\al)Q(\al),
\]
where
\[
Q(\al)=-48-294\al-157\al^2+480\al^3+492\al^4+192\al^5.
\]
Thus the difference is divisible by the shared hook, and both coefficients
take the nonzero value $-8064$ at $\al=-1$.
This is a genuinely shifted case, since
$|\lambda|=|\tilde\lambda|=11<|\mu|+|\nu|=12$.
\end{example}
\end{minipage}\par\medskip

For Macdonald coefficients, \cite[Conjecture~6.9]{Mickler:2026aa} proposed a
nonzero constant ratio along a hook locus after a power-of-$t$ normalization.
We prove a precisely normalized congruence analogue, extended to shifted
coefficients at every level.  Set $n(\lambda)=\sum_i(i-1)\lambda_i$ and
$n(\sigma)=n(\mu)+n(\nu)+n(\lambda)$ for $\sigma=(\mu,\nu,\lambda)$.

\begin{theorem}[Macdonald pivot congruences]\label{thm:mac-pivots-intro}
The normalized shifted Macdonald Littlewood--Richardson coefficients satisfy
$g^\lambda_{\mu\nu}(q,t)\in\QQ[q^{\pm},t^{\pm}]$ for all partitions
$\mu,\nu,\lambda$.  For every pivot of triples
$\sigma=(\mu,\nu,\lambda)\sim_p
\tilde\sigma=(\tilde\mu,\tilde\nu,\tilde\lambda)$ with shared hook
$h=1-q^At^B$, one has
\[
t^{-n(\sigma)}g^\lambda_{\mu\nu}(q,t)
\equiv t^{-n(\tilde\sigma)}g^{\tilde\lambda}_{\tilde\mu\tilde\nu}(q,t)
\pmod h
\qquad\text{in }\QQ[q^{\pm},t^{\pm}].
\]
\end{theorem}
Where their common restriction to $h=0$ is nonzero, the two coefficients
normalized as in this theorem have ratio $1$.  At a common zero, the
congruence alone does not determine vanishing orders or ratios of leading
terms.  We therefore do not claim equivalence with the earlier limiting-ratio
conjecture, whose normalization and hypotheses differ.

The Macdonald proof has two additional ingredients.  Knop's inversion formula
\cite[Corollary~5.2]{Knop:1997ac} gives the Laurent shift isomorphism
(Theorem~\ref{thm:mac-laurent}).  The adjoint construction needed for input
pivots uses the kernel formula of Ben Dali--D'Adderio \cite{Dali:2026aa} and
the integrality of the Bergeron--Garsia--Haiman--Tesler operators
\cite{BGHT:1999} (Proposition~\ref{prop:mac-skew-laurent}).

Section~\ref{sec:congruence} establishes the coordinate congruence, and
Sections~\ref{sec:proof} and~\ref{sec:macdonald} prove the Jack and Macdonald
theorems.  Section~\ref{sec:gkm}
realizes the Jack coefficients as equivariant cohomology classes in the Hilbert scheme of points in $\BC^2$, and we demonstrate the congruences follow from Goresky--Kottwitz--MacPherson (GKM) theory.
Section~\ref{sec:ktheory} gives the Macdonald interpretation in equivariant
$K$-theory and identifies the determinant twist behind its normalization.

\section{Congruences}\label{sec:congruence}

\subsection{Definitions}

We use definitions and notation from Macdonald \cite{Macdonald:1995}.  Throughout, rows and columns are indexed from $1$.  For a box $s=(i,j)$ in a partition $\lambda$, write
\[
\arm_\lambda(s)=\lambda_i-j,
\qquad
\leg_\lambda(s)=\lambda'_j-i.
\]
The \emph{lower} and \emph{upper} $\alpha$-hook lengths of a box
$s$ in a partition $\lambda$ are,
\[
h_\lambda(s)=\al\,\arm_\lambda(s)+\leg_\lambda(s)+1,
\qquad
h'_\lambda(s)=\al\bigl(\arm_\lambda(s)+1\bigr)+\leg_\lambda(s).
\]
The two hook products are
\[
H_\lambda=\prod_{s\in\lambda} h_\lambda(s),
\qquad
H'_\lambda=\prod_{s\in\lambda}h'_\lambda(s).
\]

Let $\Lam$ be the ring of ordinary symmetric functions in variables $x_i$. For an indeterminate $\alpha$, we let $\Lam_{[\alpha]} = \Lam \otimes \QQ[\alpha]$, let $\Lam_{[\alpha^\pm]}:=\Lam \otimes \QQ[\alpha^\pm]$ denote symmetric functions with Laurent coefficients in $\alpha$, and let $\Lamfa:=\Lam\otimes\QQ(\al)$.
We equip $\Lam_{[\alpha^\pm]}$ with the usual scalar product
determined in the power sum basis as
\begin{equation}\label{def:metricalpha}
\langle p_\rho,p_\sigma\rangle_\al=\delta_{\rho\sigma}z_\rho\al^{\ellp(\rho)}.
\end{equation}

The integral Jack symmetric functions $J_\lambda$, each defined for a partition $\lambda$ (Macdonald \cite[Ch.~VI]{Macdonald:1995}), are the unique symmetric functions with the following properties
\begin{enumerate}
\item (orthogonality) $\langle\J_\lambda,\J_\mu\rangle_\al = \delta_{\lambda\mu} j_\lambda$, where $j_\lambda := \langle\J_\lambda,\J_\lambda\rangle_\al = H_\lambda H'_\lambda$.
\item (triangularity) $\J_\mu = \sum_{\nu \le_d\, \mu} a_{\mu\nu} m_\nu$. 
\item (normalization) $[m_{1^n}]\J_\mu = n!$ for $n=|\mu|$, or equivalently
$[m_\mu]\J_\mu=H_\mu$.  Lapointe--Vinet \cite{LapointeVinet:1995} and Knop--Sahi
\cite{Knop:1997aa} proved that then
$a_{\mu\nu} \in \ZZ[\alpha]$, hence $J_\lambda \in \LamP$.
\end{enumerate}

For $\mu \subseteq \lambda$ the skew Jack function $J_{\lambda/\mu}$ is defined by adjunction, 
\begin{equation}\label{eq:shift-jack_adj}
\langle J_{\lambda/\mu}, J_\nu \rangle_\alpha = \langle J_\lambda, \J_\mu J_\nu \rangle_\alpha.
\end{equation}
By
\cite[Ch.~VI, \S~10]{Macdonald:1995}, we have $J_{\lambda/\mu} \in\LamA$.

The (integral) Jack Littlewood--Richardson coefficients $c_{\mu\nu}^{\lambda} \in \QQ(\alpha)$ are defined by
\[ J_\mu J_\nu = \sum_{\lambda: |\lambda|=|\mu|+|\nu|} c_{\mu\nu}^{\lambda } J_\lambda. \]

\subsubsection{Shifted Symmetric functions}

For shifted symmetric functions, we follow
\cite[Section~2.2]{Lassalle:2008aa}, Okounkov--Olshanski \cite{Okounkov:1996},
and Knop--Sahi \cite{Knop:1996}, specifying the Laurent coefficient ring
needed here.  Let
$\LamsA(N)$ be the algebra of polynomials in $x_1,\dots,x_N$ with coefficients in
$\QQ[\al^\pm]$ that are symmetric in the shifted variables
\[
  z_i:=\al x_i-i,\qquad 1\le i\le N.
\]
Setting the last variable to zero defines filtered algebra homomorphisms
$\LamsA(N+1)\to\LamsA(N)$, $f\mapsto f|_{x_{N+1}=0}$, and the ring of
$\al$-shifted symmetric functions is their inverse limit in the category of
filtered algebras,
\begin{equation}\label{def:shifted-symm}
\LamsA=\varprojlim_N\LamsA(N)
=\bigl\{\,(f_N)_{N\ge1}\ :\ f_{N+1}|_{x_{N+1}=0}=f_N,\ \sup_N\deg f_N<\infty\,\bigr\}.
\end{equation}
An element $f\in\LamsA$ can be evaluated at any partition $\lambda$, by setting
$x_i=\lambda_i$.  We let $\Lamsfa = \LamsA \otimes \QQ(\alpha)$ be the
shifted symmetric functions with rational coefficients in $\alpha$.
For example consider the two functions
\[ f_1  = \frac{1}{\alpha} \sum_{i\ge1} \bigl(z_i^2-i^2\bigr) \in \LamsA , \qquad f_2 = \frac{1}{\alpha+1} \sum_{i\ge1} \bigl(z_i^2-i^2\bigr) \in \Lamsfa .\]
The subtraction of $i^2$ makes each sum finite once $x_i=0$ for large $i$.  Only
$f_1$ lies in $\LamsA$, as its coefficients are genuinely Laurent.

Let $\Js_\mu$ denote the integral shifted Jack function, a rescaling of the
(monic) shifted Jack polynomials of \cite{Okounkov:1996,Knop:1996}, defined as follows:
For each partition $\mu$, $\Js_\mu$ is the unique $\alpha$-shifted symmetric function
of degree at most $|\mu|$
characterized by the interpolation conditions \cite[Section~3]{Alexandersson:2024aa}
\begin{equation}\label{eq:shifted-jack-interpolation}
\Js_\mu(\lambda)=
\begin{cases}
\al^{-|\mu|}H_\mu H'_\mu, & \lambda=\mu,\\
0, & |\lambda|\le |\mu|,\ \lambda\ne\mu.
\end{cases}
\end{equation}
The top homogeneous component of $\Js_\mu$ is equal to the ordinary
Jack polynomial $J_\mu$ in the corresponding normalization.
The integral shifted Jack functions have Laurent-polynomial coefficients,
$\Js_\lambda\in\LamsA$: we have $\Js_\lambda=\sh(\J_\lambda)$ with
$\J_\lambda\in\LamP$, so this follows from Corollary~\ref{cor:jack-laurent-iso}
below (it is also a consequence of \cite[Section~6]{Knop:1997ac} and
\cite[Theorem~A]{Naqvi:2023}). Knop--Sahi \cite[Section~5]{Knop:1996} prove the higher vanishing condition that $\Js_\mu(\lambda)= 0$ unless $\mu\subseteq\lambda$.

The shifted Jack Littlewood--Richardson coefficients $d^\lambda_{\mu\nu}\in \QQ(\alpha)$ are defined by \cite[Section~4]{Alexandersson:2024aa} (see also Sahi
\cite{Sahi:2010}, who studies the corresponding coefficients of the monic
interpolation polynomials)
\[
\Js_\mu\Js_\nu=\sum_{\lambda:|\lambda|\le|\mu|+|\nu|} d^\lambda_{\mu\nu}\Js_\lambda.
\]
The normalized coefficient appearing in the congruence conjectures is
\[
g^\lambda_{\mu\nu}:=H_\lambda H'_\lambda d^\lambda_{\mu\nu}.\]
This normalization clears the rational
denominators of $d^\lambda_{\mu\nu}$, and Alexandersson--F\'eray \cite[Prop. 2]{Alexandersson:2024aa} show that for nonempty $\mu,\nu$ the $g^\lambda_{\mu\nu}$ are
Laurent polynomials in $\al$,
\[
\alpha^{|\mu|+|\nu|-|\lambda|-2} g^\lambda_{\mu\nu}\in\QQ[\al]
\]
(for $\mu=\emptyset$ one has $g^\lambda_{\emptyset\nu}=\delta_{\lambda\nu}H_\nu H'_\nu$, which can violate this bound, e.g.\ $g^{(1)}_{\emptyset,(1)}=\al$).
Negative powers of $\al$ do occur in $g$. For example
\[ g^{2111}_{2111,1111} = \alpha^{-1}
144(\alpha + 1)^2(\alpha + 2)^2(\alpha + 4)(2\alpha + 3)^2. \]
When $|\lambda|=|\mu|+|\nu|$, we have $g^\lambda_{\mu\nu} = \langle\J_\mu\J_\nu,\J_\lambda\rangle_\al \in \ZZ[\alpha]$.

%

\subsection{Shifted coordinate congruence}

The shifted variables $z_i=\al x_i-i$ are evaluated at a partition $\lambda = (\lambda_1, \lambda_2, \ldots)$ as
\[
z_i(\lambda)=\al\lambda_i-i.
\]
The argument begins with the following relation between pivots and shifted coordinates.
\begin{lemma}\label{lem:shifted-coordinate-pivot}
If $\lambda \sim_p \tilde\lambda$ via a pivot at $p$ with common hook $h$, removing a box from row
$a$ and adding a box to row $b>a$, then
\[
\{z_i(\lambda)\}_{i\ge1}\equiv \{z_i(\tilde\lambda)\}_{i\ge1}\pmod h
\]
as multisets, after the interchange of the two elements indexing the affected rows
$a\leftrightarrow b$.
\end{lemma}
\begin{proof}
By assumption,
\[
\tilde\lambda_a=\lambda_a-1,\qquad
\tilde\lambda_b=\lambda_b+1,
\]
and all other parts are unchanged.  The pivot column is $j=\lambda_b+1$.  At the
pivot box $p=(a,j)$ in $\lambda$, we have
\[
\arm_\lambda(p)=\lambda_a-\lambda_b-1,
\qquad
\leg_\lambda(p)=b-a-1.
\]
Thus the shared hook is
\[
h=h_\lambda(p)=h'_{\tilde\lambda}(p)
=\al(\lambda_a-\lambda_b-1)+(b-a).
\]

Only the shifted coordinates in rows $a$ and $b$ change.  A direct calculation
gives
\[
z_a(\lambda)-z_b(\tilde\lambda)
=\al\lambda_a-a-\bigl(\al(\lambda_b+1)-b\bigr)
=\al(\lambda_a-\lambda_b-1)+(b-a)
=h,
\]
and
\[
z_b(\lambda)-z_a(\tilde\lambda)
=\al\lambda_b-b-\bigl(\al(\lambda_a-1)-a\bigr)
=-h.
\]
For every row $r\notin\{a,b\}$ we have
$z_r(\lambda)=z_r(\tilde\lambda)$.  Hence the two shifted-coordinate multisets
agree modulo $h$ after interchanging the two affected rows $a$ and $b$.
\end{proof}


\subsection{Function congruence}

The following lemma shows that a pivot congruence holds for
\emph{every} Laurent $\al$-shifted symmetric function.

\begin{lemma}\label{lem:regular-pivot}
Let $\lambda\sim_p\tilde\lambda$ be a pivot with shared hook $h$.  For any
$\al$-shifted symmetric function $F\in\LamsA$, we have
\[
F(\lambda)\equiv F(\tilde\lambda)\pmod h.
\]
Equivalently $F(\lambda)-F(\tilde\lambda)$
is divisible by $h$ in $\QQ[\al^\pm]$.
\end{lemma}
\begin{proof}
Let $a<b$ be the rows where a box is removed and added, and choose $N$ at least
the number of parts of $\lambda$ and of $\tilde\lambda$.  Then $F(\lambda)$ and
$F(\tilde\lambda)$ are computed by the image $F_N\in\LamsA(N)$ of $F$, a symmetric
polynomial in $z_1,\ldots,z_N$ with coefficients in $\QQ[\al^\pm]$.  Write $s_{ab}$
for the transposition of the $a$-th and $b$-th coordinates and $e_a,e_b$ for the
standard basis vectors.  By Lemma~\ref{lem:shifted-coordinate-pivot},
\[
z(\lambda)=s_{ab}\,z(\tilde\lambda)+h\,(e_a-e_b).
\]
For any ring $R$, any $f\in R[z_1,\ldots,z_N]$, any $z,v\in R^N$ and any
$\delta\in R$,
\begin{equation}\label{eq:modprop}
f(z+\delta v) - f(z) \in \delta R,
\end{equation}
since the difference is a polynomial multiple of $\delta$.  Taking
$R=\QQ[\al^\pm]$, $\delta=h$, $v=e_a-e_b$ and using the symmetry of $F_N$,
\[
F(\lambda)-F(\tilde\lambda)
=F_N\bigl(s_{ab}z(\tilde\lambda)+h(e_a-e_b)\bigr)-F_N\bigl(s_{ab}z(\tilde\lambda)\bigr)
\in h\,\QQ[\al^\pm].
\]
The congruence is not vacuous.  The units of $\QQ[\al^\pm]$ are the nonzero scalar
multiples of powers of $\al$, and $h=\al x+y$ with $x,y\ge1$ is not of this form.
\end{proof}

\begin{example}
The integral shifted Jack function $\Js_\mu$ is $\al$-shifted symmetric, so by
Lemma~\ref{lem:regular-pivot}, for a pivot $\tilde\mu\sim_p\mu$ we must have
\[
\Js_\mu(\mu)\equiv \Js_\mu(\tilde\mu)\pmod h.
\]
By the defining vanishing property of shifted Jacks,
Equation~\eqref{eq:shifted-jack-interpolation}, this is
\[
\al^{-|\mu|}H_\mu H'_\mu\equiv 0\pmod h.
\]
Since $h$ is always one of the hook factors in either $H_\mu$ or $H'_\mu$, the
congruence holds.
\end{example}

\section{Proof of the Jack congruences}\label{sec:proof}

By Lemma~\ref{lem:regular-pivot}, every Laurent $\al$-shifted symmetric
function satisfies the pivot congruence.  For \emph{output} pivots, in which
$\lambda$ changes, we realize $g^\lambda_{\mu\nu}$ as a unit times the value
at $\lambda$ of such a function, constructed from $\Phi_{\mu\nu}$.
For \emph{input} pivots, in which $\nu$ changes, the corresponding evaluation
at $\nu$ comes from the adjoint construction $\Psi_{\lambda/\mu}$.
The Laurent shift isomorphism supplies the output construction, while the input
construction also uses the fact that the Jack pairing is perfect in each
degree over $\QQ[\al^{\pm}]$.  Symmetry in $\mu,\nu$ covers the other input direction.

\subsection{Strategy}

For output pivots, the natural starting point is the product
\[
\Js_\mu\Js_\nu=\sum_{\rho:\,|\rho|\le|\mu|+|\nu|}
d^\rho_{\mu\nu}\Js_\rho.
\]
To isolate $g^\lambda_{\mu\nu}$ for $|\lambda|=k$, take the \emph{size-$k$
slice in the shifted-Jack expansion},
\begin{equation}\label{deg-k-trunc}
F^{(k)}:=\sum_{|\rho|=k}d^\rho_{\mu\nu}\Js_\rho.
\end{equation}
The interpolation conditions leave one term at $\lambda$:
\[
F^{(k)}(\lambda)=d^\lambda_{\mu\nu}\Js_\lambda(\lambda)
=\al^{-k}g^\lambda_{\mu\nu}.
\]
Thus Lemma~\ref{lem:regular-pivot} applies if $F^{(k)}$ is Laurent.

This slice is not the ordinary homogeneous degree-$k$ component of
$\Js_\mu\Js_\nu$: the shifted ring is only \emph{filtered}, and each
$\Js_\rho$ contains terms of lower degree.  Laurent coefficients of the
product do not immediately imply Laurent coefficients of the slice, since
the individual $d^\rho_{\mu\nu}$ are rational in $\al$.
We instead form $\Phi_{\mu\nu}=\sh^{-1}(\Js_\mu\Js_\nu)$ in the
\emph{graded} ordinary symmetric-function ring, project to its homogeneous
component $\Phi^{(k)}_{\mu\nu}$, and shift back:
$F^{(k)}=\sh(\Phi^{(k)}_{\mu\nu})$.
The Laurent isomorphism proved next ensures that each step preserves the
Laurent lattice.

\subsection{Lassalle's shift transform}

We can move between the filtered ring of shifted symmetric functions and the graded ring of symmetric functions using Lassalle's shift transform.

\begin{proposition}[Lassalle {\cite[Section~3]{Lassalle:2008aa}}]\label{prop:lassalleshift}
There exists a $\QQ(\alpha)$-linear
isomorphism
\[
\sh:\Lamfa\;\xrightarrow{\ \sim\ }\;\Lamsfa
\]
from the ring $\Lamfa$ of ordinary symmetric functions, 
onto the ring
$\Lamsfa$ of $\al$-shifted symmetric functions.

This map is characterized by the identity
\begin{equation}\label{eq:lassalle-shift}
\sh(f)(\lambda)=\al^{-|\lambda|}\langle e^{p_1}f,\J_\lambda\rangle_\al
\end{equation}
for every $f\in\Lamfa$ and every partition $\lambda$.
Furthermore, it satisfies $\sh(\J_\rho)=\Js_\rho$.
\end{proposition}
For our purposes, we need $\sh$ to restrict to an isomorphism between the lattices
with Laurent $\QQ[\al^\pm]$ coefficients (Corollary~\ref{cor:jack-laurent-iso}
below).  The map $\sh$ is linear but not multiplicative.

Recall that $\LamsA$ carries the Laurent basis of \emph{stable shifted power
sums}
\[
p^\circ_k=\sum_i\bigl(z_i^k-(-i)^k\bigr),\qquad
P_\nu=\prod_j p^\circ_{\nu_j},\qquad z_i=\al x_i-i,
\]
which are manifestly $\al$-shifted symmetric with Laurent coefficients and form
a $\QQ[\al^\pm]$-basis $\{P_\nu\}$ of $\LamsA$.  The top homogeneous component of
$P_\nu$ in the variables $x_i$ is $\al^{|\nu|}p_\nu$.

A shifted symmetric function with rational coefficients is Laurent as soon as all
of its values are.

\begin{lemma}\label{lem:laurent-values}
Let $F\in\Lamsfa$.  If $F(\lambda)\in\QQ[\al^\pm]$ for every partition
$\lambda$, then $F\in\LamsA$.
\end{lemma}
\begin{proof}
Write $F=D^{-1}\sum_\nu a_\nu P_\nu$, a finite sum with $a_\nu,D\in\QQ[\al]$,
choosing this representation with $D$ of minimal degree.  Then no irreducible
factor of $D$ divides every $a_\nu$.  Suppose that some irreducible
$\pi\in\QQ[\al]$, not a scalar multiple of $\al$, divides $D$, and let
$\al_0\in\overline{\QQ}$ be a root of $\pi$, so $\al_0\ne0$.  The function $G=\sum_\nu a_\nu P_\nu$ satisfies
$G(\lambda)=D\,F(\lambda)\in\pi\,\QQ[\al^\pm]$, so $G(\lambda)$ vanishes at
$\al=\al_0$ for every $\lambda$.  Fix $N\ge\max|\nu|$ and consider the partitions
with at most $N$ parts.  At $\al=\al_0$ the point
$z(\lambda)=(\al_0\lambda_i-i)_{i\le N}$ is the image of $\lambda$ under an
invertible affine map of $\overline{\QQ}^N$, and these partitions are Zariski dense
in $\overline{\QQ}^N$ (substitute $\lambda_i=m_i+m_{i+1}+\dots+m_N$ with the
$m_j\ge0$ arbitrary).  Hence the polynomial
$\sum_\nu a_\nu(\al_0)P_\nu(z_1,\dots,z_N)$ vanishes identically.  Here
$P_\nu=\prod_j\bigl(p_{\nu_j}(z)-c_{\nu_j}\bigr)$ with constants
$c_k=\sum_{i\le N}(-i)^k$, and $p_1(z),\dots,p_N(z)$ are algebraically
independent, so the $P_\nu$ with $|\nu|\le N$ are linearly independent.  Thus
$a_\nu(\al_0)=0$, i.e.\ $\pi\mid a_\nu$, for every $\nu$, a contradiction.  So
$D=c\,\al^k$ for some $c\in\QQ^\times$ and $k\ge0$, and
$F=c^{-1}\al^{-k}\sum_\nu a_\nu P_\nu\in\LamsA$.
\end{proof}

\begin{proposition}\label{prop:jack-characters-laurent}
The Jack characters are Laurent shifted symmetric functions:
$\sh(p_\mu)\in\LamsA$ for every partition $\mu$.
\end{proposition}
\noindent Polynomiality properties of Jack characters, in a different
normalization, are proved by Do{\l}\k{e}ga--F\'eray \cite{DolegaFeray:2016}.
We include the short argument we need.
\begin{proof}
By Proposition~\ref{prop:lassalleshift}, $\sh(p_\mu)\in\Lamsfa$, and by
\eqref{eq:lassalle-shift} its value at $\lambda$ is
$\al^{-|\lambda|}\langle e^{p_1}p_\mu,\J_\lambda\rangle_\al$.  By Knop--Sahi
\cite{Knop:1997aa}, $\J_\lambda\in\LamP$, so its power-sum coefficients lie in
$\QQ[\al]$.  Since $e^{p_1}p_\mu$ has rational coefficients and
$\langle p_\rho,p_\sigma\rangle_\al\in\QQ[\al]$, the value lies in
$\al^{-|\lambda|}\QQ[\al]\subset\QQ[\al^\pm]$.  Lemma~\ref{lem:laurent-values}
completes the proof.
\end{proof}

\begin{corollary}\label{cor:jack-laurent-iso}
The shift map restricts to an isomorphism of $\QQ[\al^\pm]$-lattices
\[
\sh:\LamA\;\xrightarrow{\ \sim\ }\;\LamsA .
\]
\end{corollary}
\begin{proof}
By Proposition~\ref{prop:jack-characters-laurent}, $\sh$ carries the power-sum
basis of $\LamA$ into $\LamsA$.  The shift map preserves the degree filtration
and its top-degree part is the identity, so
$\sh(p_\mu)=p_\mu+(\text{lower degree})$.  As the top homogeneous component of
$P_\mu$ is $\al^{|\mu|}p_\mu$, the transition matrix
$R_{\mu\nu}:=[P_\nu]\sh(p_\mu)$, for $\mu,\nu$ of degree at most any fixed $k$,
is triangular with respect to degree with diagonal entries $\al^{-|\mu|}$.  These
are units of $\QQ[\al^\pm]$, so the inverse matrix is again defined over
$\QQ[\al^\pm]$.
\end{proof}
In the Macdonald case the analogue of Corollary~\ref{cor:jack-laurent-iso}
follows from Knop's inversion formula (Theorem~\ref{thm:mac-laurent}).

\begin{corollary}\label{cor:jack-pivot}
Let $\lambda\sim_p\tilde\lambda$ be a pivot with shared hook $h$.  Then
$\J_\lambda\equiv\J_{\tilde\lambda}\pmod h$, coefficientwise in the power-sum
basis.
\end{corollary}
\begin{proof}
Let $|\mu|=|\lambda|=n$.  Only the degree-$n$ term of $e^{p_1}p_\mu$ pairs
nontrivially with $\J_\lambda$, so \eqref{eq:lassalle-shift} gives
$[p_\mu]\J_\lambda=z_\mu^{-1}\al^{n-\ell(\mu)}\sh(p_\mu)(\lambda)$.  By
Proposition~\ref{prop:jack-characters-laurent} and
Lemma~\ref{lem:regular-pivot}, $[p_\mu]\J_\lambda\equiv[p_\mu]\J_{\tilde\lambda}$
modulo $h$ in $\QQ[\al^\pm]$, hence in $\QQ[\al]$, as $h$ is coprime to $\al$.
\end{proof}
For $|\lambda|=|\mu|+|\nu|$ this already gives Theorem~\ref{thm:top-pivots},
since $g^\lambda_{\mu\nu}=\langle\J_\mu\J_\nu,\J_\lambda\rangle_\al$ and the
pairing is polynomial in $\al$ on power sums, while the shifted coefficients at
lower levels need the constructions below.

\subsection{Output pivots}

\begin{figure}[ht]
\centering
\begin{tikzpicture}[>=stealth,scale=1.0]
  \draw[gray!50] (0,0) -- (-3.12,4.9);
  \draw[gray!50] (0,0) -- (3.12,4.9);
  \draw[gray!50] (0,4.9) ellipse [x radius=3.12, y radius=0.60];
  \draw[thick] (-0.51,1.2) -- (-1.8,4.2);
  \draw[thick] (0.51,1.2) -- (1.8,4.2);
  \draw[thick] (0,4.2) ellipse [x radius=1.8, y radius=0.40];
  \draw[thick] (0,2.4) ellipse [x radius=1.03, y radius=0.23];
  \draw[thick] (0,1.2) ellipse [x radius=0.51, y radius=0.12];
  \node at (-1.2,2.0) {$\Phi_{\mu\nu}$};
  \draw[->] (-3.2,-0.3) -- (-3.2,5.3) node[above] {\footnotesize degree};
  \draw[dashed] (-3.2,4.2) -- (1.8,4.2);
  \draw[dashed] (-3.2,2.4) -- (1.03,2.4);
  \draw[dashed] (-3.2,1.2) -- (0.51,1.2);
  \draw[dashed] (-3.2,0) -- (0,0);
  \node[left] at (-3.2,4.2) {$|\mu|+|\nu|$};
  \node[left] at (-3.2,2.4) {$k$};
  \node[left] at (-3.2,1.2) {$\max(|\mu|,|\nu|)$};
  \node[left] at (-3.2,0) {$0$};
  \draw[->] (4.6,4.2) -- (1.85,4.2);
  \node[right] at (4.6,4.2) {$\Phi^{(|\mu|+|\nu|)}_{\mu\nu}=\J_\mu\J_\nu$};
  \draw[->] (4.6,2.4) -- (1.08,2.4);
  \node[right] at (4.6,2.4) {$\Phi^{(k)}_{\mu\nu}$};
  \draw[->] (4.6,1.2) -- (0.56,1.2);
  \node[right] at (4.6,1.2) {$\Phi^{(\max(|\mu|,|\nu|))}_{\mu\nu}$};
\end{tikzpicture}
\caption{The graded symmetric function
$\Phi_{\mu\nu}=\operatorname{sh}^{-1}(\Js_\mu\Js_\nu)$ has top degree
$|\mu|+|\nu|$, where its leading slice is the ordinary product $\J_\mu\J_\nu$.
By Knop--Sahi vanishing $\Phi^{(k)}_{\mu\nu}=0$ for $k<\max(|\mu|,|\nu|)$, so
its support is the band $\max(|\mu|,|\nu|)\le k\le|\mu|+|\nu|$.}
\label{fig:phi-cone}
\end{figure}

\begin{proposition}\label{prop:output-pivots}
Let $\lambda\sim_p\tilde\lambda$ be a pivot with shared
hook $h$.  Then
\[
g^\lambda_{\mu\nu}\equiv g^{\tilde\lambda}_{\mu\nu}\pmod h .
\]
\end{proposition}
\begin{proof}
Recall from \eqref{eq:shifted-jack-interpolation} and its sequel that
$\Js_\mu\Js_\nu=\sum_\rho d^\rho_{\mu\nu}\Js_\rho$ with
$g^\rho_{\mu\nu}=H_\rho H'_\rho\,d^\rho_{\mu\nu}$.  We first record that the
relevant combination of these constants is a Laurent symmetric function, even
though the individual $d^\rho_{\mu\nu}$ may have hook denominators.  Indeed,
$\Js_\mu\Js_\nu\in\LamsA$, since each factor is Laurent
($\Js_\rho=\sh(\J_\rho)$ with $\J_\rho\in\LamP$, so
Corollary~\ref{cor:jack-laurent-iso} applies), and the inverse transform $\sh^{-1}$ is a
$\QQ[\al^\pm]$-linear isomorphism
$\LamsA\to\LamA$ with $\sh^{-1}(\Js_\rho)=\J_\rho$, so
\[
\Phi_{\mu\nu} :=\sh^{-1}(\Js_\mu\Js_\nu)=\sum_{\rho: |\rho|\leq |\mu|+|\nu|} d^\rho_{\mu\nu}\,\J_\rho
\ \in\ \LamA :
\]
the hook denominators of the individual $d^\rho_{\mu\nu}$ cancel in the sum, even
though each $d$ is in general only a rational function in $\QQ(\al)$ (see
Example~\ref{ex:laurent-theta}).
Let
\[
\Phi_{\mu\nu}^{(k)}:=\sum_{|\rho|=k} d^\rho_{\mu\nu}\,\J_\rho
\]
be the homogeneous degree-$k$ component of $\Phi_{\mu\nu}$
(see Figure~\ref{fig:phi-cone}). Since
$\Phi_{\mu\nu}\in\LamA$
and each $\J_\rho$ is homogeneous of degree $|\rho|$, we have
$\Phi_{\mu\nu}^{(k)}\in\LamA$ as well.  By \eqref{eq:lassalle-shift} and the
orthogonality $\langle\J_\rho,\J_\sigma\rangle_\al=\delta_{\rho\sigma}H_\rho H'_\rho$, for every
$\lambda$ with $|\lambda|=k$ we have
\[
\sh(\Phi_{\mu\nu}^{(k)})(\lambda)
=\al^{-k}\langle\Phi_{\mu\nu}^{(k)},\J_\lambda\rangle_\al
=\al^{-k}\,d^\lambda_{\mu\nu}H_\lambda H'_\lambda
=\al^{-k}\,g^\lambda_{\mu\nu}.
\]
Thus $\sh(\Phi_{\mu\nu}^{(k)})\in\LamsA$ is a single $\al$-shifted symmetric function whose
values at the partitions of size $k$ are, up to the unit $\al^{-k}$, the
coefficients $g^\lambda_{\mu\nu}$.  Applying Lemma~\ref{lem:regular-pivot} to
$\sh(\Phi_{\mu\nu}^{(k)})$ at the pivot $\lambda\sim_p\tilde\lambda$ (note
$|\lambda|=|\tilde\lambda|=k$) gives
\[
\al^{-k}\,g^\lambda_{\mu\nu}\equiv \al^{-k}\,g^{\tilde\lambda}_{\mu\nu}\pmod h,
\]
and since $\al^{-k}$ is a unit of $\QQ[\al^\pm]$, the factor cancels, yielding
$g^\lambda_{\mu\nu}\equiv g^{\tilde\lambda}_{\mu\nu}\pmod h$.
\end{proof}

\begin{example}\label{ex:laurent-theta}
Here we show that the Laurentness of $\Phi_{\mu\nu}=\sum_\rho d^\rho_{\mu\nu}\J_\rho$ is a genuine
cancellation. The individual terms in the sum carry hook-line denominators that disappear upon summation.  Take $\mu=\nu=(1^3)$ and the degree-$4$ component $\Phi_{1^3,1^3}^{(4)}$,
which has two terms, indexed by $\rho=(1^4)$ and $\rho=(2,1,1)$. The two
non-zero coefficients $d^\rho$ are
\[
d^{(1^4)}_{(1^3),(1^3)}=\frac{18(\al+1)(\al+2)}{\al(\al+3)},\qquad
d^{(2,1,1)}_{(1^3),(1^3)}=\frac{18(\al+1)}{\al(\al+3)},
\]
each carrying the hook-line denominator $\al+3$.  In the power-sum basis both
$d^\rho\J_\rho$ have $\al+3$ in their denominators, but it cancels in the sum:
\[
\Phi_{1^3,1^3}^{(4)}=d^{(1^4)}_{(1^3),(1^3)}\J_{(1^4)}+d^{(2,1,1)}_{(1^3),(1^3)}\J_{(2,1,1)}
=\frac{18(\al+1)}{\al}\bigl(p_1^4-5\,p_1^2p_2+2\,p_2^2+6\,p_1p_3-4\,p_4\bigr).
\]
\end{example}

\subsection{Input pivots and the general congruence}
For an \emph{input} pivot we keep $\lambda,\mu$ fixed and pivot in $\nu$.  The relevant
object is now the adjoint of the
$\QQ(\al)$-linear operator $L_\mu(f):=\sh^{-1}(\Js_\mu\,\sh(f))$, which satisfies
$L_\mu(\J_\nu)=\Phi_{\mu\nu}$.  Define $\Psi_{\lambda/\mu}\in\Lamfa$ by
\begin{equation}\label{eq:skew-sum-def}
\langle\Psi_{\lambda/\mu},\J_\tau\rangle_\al
=\langle\J_\lambda,L_\mu(\J_\tau)\rangle_\al
=g^\lambda_{\mu\tau}
\qquad\text{for every }\tau ,
\end{equation}
the last equality from $d^\lambda_{\mu\tau}\langle\J_\lambda,\J_\lambda\rangle_\al
=g^\lambda_{\mu\tau}$.  Equivalently
\[
\Psi_{\lambda/\mu}=\sum_{\tau \subseteq \lambda}
\frac{g^\lambda_{\mu\tau}}{\langle\J_\tau,\J_\tau\rangle_\al}\,\J_\tau ,
\]
a finite sum, since $g^\lambda_{\mu\tau}=0$ unless $\tau\subseteq\lambda$.
This support property follows by evaluating the shifted product successively
at partitions of increasing size and using Knop--Sahi extra vanishing
\cite[Section~5]{Knop:1996}.  Its homogeneous components
$\Psi_{\lambda/\mu}^{(m)}$ range over $|\lambda|-|\mu|\le m\le|\lambda|$. The
bottom component is the ordinary skew Jack
$\Psi_{\lambda/\mu}^{(|\lambda|-|\mu|)}=\J_{\lambda/\mu}$ governing the top-degree
input pivot, while the top component, supported on the single
term $\tau=\lambda$, records the shifted-Jack evaluation
$d^\lambda_{\mu\lambda}=\Js_\mu(\lambda)$, proportional to the generalized Jack
binomial coefficient $\binom{\lambda}{\mu}_{\al}:=\Js_\mu(\lambda)/\Js_\mu(\mu)$
of \cite{Okounkov:1996,Lassalle:2008aa}, see Figure~\ref{fig:psi-cone}.

\begin{lemma}\label{lem:skew-laurent}
For all partitions $\mu\subseteq\lambda$, $\Psi_{\lambda/\mu}\in\LamA$.
\end{lemma}
\begin{proof}
The operator $L_\mu(f)=\sh^{-1}(\Js_\mu\,\sh(f))$ preserves $\LamA$, as by
Corollary~\ref{cor:jack-laurent-iso} both $\sh$ and $\sh^{-1}$ preserve the
Laurent lattice, and multiplication by $\Js_\mu\in\LamsA$ preserves the ring
$\LamsA$.  Writing
$\Psi_{\lambda/\mu}=\sum_\rho b_\rho p_\rho$, the adjunction
\eqref{eq:skew-sum-def} extended $\QQ(\al)$-linearly gives
\[
b_\rho\,\langle p_\rho,p_\rho\rangle_\al
=\langle\Psi_{\lambda/\mu},p_\rho\rangle_\al
=\langle\J_\lambda,L_\mu(p_\rho)\rangle_\al .
\]
The right side lies in $\QQ[\al^\pm]$, since $\J_\lambda,L_\mu(p_\rho)\in\LamA$ and the
Jack pairing has $\QQ[\al^\pm]$-coefficients in the power-sum basis.  As
$\langle p_\rho,p_\rho\rangle_\al=z_\rho\al^{\ell(\rho)}$ is a unit of $\QQ[\al^\pm]$, each
$b_\rho\in \QQ[\al^\pm]$.
\end{proof}

\begin{definition}\label{def:shifted-skew}
Following Lemma~\ref{lem:skew-laurent}, define the \emph{shifted skew Jack function}\footnote{The subscript $\lambda/\mu$ is notational. It records
the ordered pair $\mu\subseteq\lambda$ and the fact that the leading
(lowest-degree) term of $\Psi_{\lambda/\mu}$ is the ordinary skew Jack
$\J_{\lambda/\mu}$.  It should \emph{not} be read as a skew shape. Like the
ordinary skew Jack itself, $\Js_{\lambda/\mu}$ depends on the pair $(\lambda,\mu)$
and not merely on the skew diagram $\lambda/\mu$.}
by
\[
\Js_{\lambda/\mu}:=\sh\bigl(\Psi_{\lambda/\mu}\bigr)
\ \in\ \LamsA.
\]
\end{definition}

\begin{example}\label{ex:skew-laurent}
Take $\mu=(1,1)$ and $\lambda=(3,1)$, so the skew degree is
$|\lambda|-|\mu|=2$. We examine the first genuinely shifted component of degree $m=3$.  Since $g^\lambda_{\mu\nu}=0$ unless $\nu\subseteq\lambda$, the
only partitions of $3$ that contribute are $\nu=(3)$ and $\nu=(2,1)$, with
\[
d^{(3,1)}_{(1,1),(3)}=6,\qquad
d^{(3,1)}_{(1,1),(2,1)}=\frac{\al+2}{(\al+1)^2}.
\]
The corresponding coefficients of $\Psi_{\lambda/\mu}^{(3)}$ in the $J$ basis,
$g^\lambda_{\mu\nu}/\langle\J_\nu,\J_\nu\rangle_\al$, are
\[
[\J_{(3)}]\,\Psi_{\lambda/\mu}^{(3)}
=\frac{4(\al+1)(3\al+1)}{2\al+1},\qquad
[\J_{(2,1)}]\,\Psi_{\lambda/\mu}^{(3)}
=\frac{4\al(3\al+1)}{2\al+1}.
\]
Yet in the component $\Psi_{\lambda/\mu}^{(3)}$ the factor $2\al+1$ cancels, and
in the power-sum basis the full component is polynomial
\[
\Psi_{\lambda/\mu}^{(3)}=(12\al+4)\,p_{111}+(24\al^2+8\al)\,p_{21}
+(12\al^3+4\al^2)\,p_3 .
\]
\end{example}

\begin{figure}[ht]
\centering
\begin{tikzpicture}[>=stealth,scale=1.0]
  \draw[gray!50] (0,-0.5) -- (-3.12,5.0);
  \draw[gray!50] (0,-0.5) -- (3.12,5.0);
  \draw[gray!50] (0,5.0) ellipse [x radius=3.12, y radius=0.60];
  \draw[thick] (0,4.6) -- (-1.0,1.8);
  \draw[thick] (0,4.6) -- (1.0,1.8);
  \draw[thick] (0,1.8) ellipse [x radius=1.0, y radius=0.24];
  \draw[thick] (0,3.2) ellipse [x radius=0.5, y radius=0.13];
  \node at (-1.3,2.6) {$\Psi_{\lambda/\mu}$};
  \draw[->] (-3.6,-0.8) -- (-3.6,5.3) node[above] {\footnotesize degree};
  \draw[dashed] (-3.6,4.6) -- (0,4.6);
  \draw[dashed] (-3.6,3.2) -- (0.5,3.2);
  \draw[dashed] (-3.6,1.8) -- (1.0,1.8);
  \draw[dashed] (-3.6,-0.5) -- (0,-0.5);
  \node[left] at (-3.6,4.6) {$|\lambda|$};
  \node[left] at (-3.6,3.2) {$k$};
  \node[left] at (-3.6,1.8) {$|\lambda|-|\mu|$};
  \node[left] at (-3.6,-0.5) {$0$};
  \draw[->] (4.9,4.6) -- (0.05,4.6);
  \node[right,align=left] at (4.9,4.6)
     {$\Psi^{(|\lambda|)}_{\lambda/\mu}=\binom{\lambda}{\mu}_{\!\al}\al^{-|\mu|}H_\mu H'_\mu \,\J_\lambda$};
  \draw[->] (4.9,3.2) -- (0.55,3.2);
  \node[right] at (4.9,3.2) {$\Psi^{(k)}_{\lambda/\mu}$};
  \draw[->] (4.9,1.8) -- (1.05,1.8);
  \node[right] at (4.9,1.8) {$\Psi^{(|\lambda|-|\mu|)}_{\lambda/\mu}=\J_{\lambda/\mu}$};
\end{tikzpicture}
\caption{The function
$\Psi_{\lambda/\mu}$ is supported in degrees $|\lambda|-|\mu|$ through
$|\lambda|$. Its \emph{bottom} slice is the ordinary skew Jack
$\J_{\lambda/\mu}$, while its \emph{top} slice is the single term
$\Js_\mu(\lambda)\,\J_\lambda$, whose coefficient is proportional to the generalized Jack
binomial coefficient $\binom{\lambda}{\mu}_{\al}=\Js_\mu(\lambda)/\Js_\mu(\mu)$
of \cite{Okounkov:1996,Lassalle:2008aa}.}
\label{fig:psi-cone}
\end{figure}

\begin{proposition}\label{prop:input-pivots}
Let $\lambda,\mu$ be fixed and let
$\nu\sim_p\tilde\nu$ be a pivot with shared hook $h$, so $|\nu|=|\tilde\nu|=m$.
Then
\[
g^\lambda_{\mu\nu}\equiv g^\lambda_{\mu\tilde\nu}\pmod h.
\]
\end{proposition}
\begin{proof}
If $\mu\not\subseteq\lambda$, both sides vanish by the extra vanishing
property, so assume $\mu\subseteq\lambda$.
By Lemma~\ref{lem:skew-laurent}, $\Psi_{\lambda/\mu}\in
\LamA$, hence so is its homogeneous component
$\Psi_{\lambda/\mu}^{(m)}$ of degree $m$, and
$\sh(\Psi_{\lambda/\mu}^{(m)})\in\LamsA$ by Corollary~\ref{cor:jack-laurent-iso}.
By \eqref{eq:lassalle-shift}, for every $|\nu|=m$,
\[
\sh\bigl(\Psi_{\lambda/\mu}^{(m)}\bigr)(\nu)
=\al^{-m}\langle\Psi_{\lambda/\mu}^{(m)},\J_\nu\rangle_\al
=\al^{-m}\,g^\lambda_{\mu\nu}.
\]
Applying Lemma~\ref{lem:regular-pivot} to the $\al$-shifted symmetric function
$\sh(\Psi_{\lambda/\mu}^{(m)})$ at the pivot $\nu\sim_p\tilde\nu$ gives
$\al^{-m}g^\lambda_{\mu\nu}\equiv\al^{-m}g^\lambda_{\mu\tilde\nu}\pmod h$, and
since $\al^{-m}$ is a unit of $\QQ[\al^\pm]$ the factor cancels.
\end{proof}

\begin{proof}[Proof of Theorem~\ref{thm:top-pivots}]
Combining Proposition~\ref{prop:output-pivots}, Proposition~\ref{prop:input-pivots},
and the symmetry $g^\lambda_{\mu\nu}=g^\lambda_{\nu\mu}$ covers all three pivot
directions, proving the theorem.
\end{proof}

\section{Extension to Macdonald functions}\label{sec:macdonald}

We now prove the Macdonald pivot congruences.  The Jack argument of
Section~\ref{sec:proof} rests on two Laurent statements: the shift map
preserves Laurent coefficients (Corollary~\ref{cor:jack-laurent-iso}), and so
does its skew analogue (Lemma~\ref{lem:skew-laurent}).  In the Macdonald case the
first follows from Knop's inversion formula (Section~\ref{sec:mac-shift}), and
the second from a kernel formula of Ben Dali--D'Adderio together with the
integrality of the Bergeron--Garsia--Haiman--Tesler operators
(Section~\ref{sec:mac-input}).  The congruences
then follow as in the Jack case (Section~\ref{sec:mac-pivots}).

\subsection{Definitions and multiplicative structure}\label{sec:mac-defs}

Throughout this section, $J_\lambda$, $\Js_\lambda$, $d^\lambda_{\mu\nu}$,
$g^\lambda_{\mu\nu}$, $\Phi_{\mu\nu}$, $\Psi_{\lambda/\mu}$, and $L_\mu$
denote the Macdonald analogues of the Jack objects above, with dependence on
$q,t$ suppressed when unambiguous.

\begin{definition}
\label{def:qt-shifted-symm}
Let $\Lamsl(N)$ be the algebra of polynomials in $x_1,\dots,x_N$ with coefficients
in $\QQ[q^{\pm},t^{\pm}]$ that are symmetric in $Y_i(x):=x_i\,t^{1-i}$,
$1\le i\le N$.  The ring of $(q,t)$-shifted symmetric functions (see e.g.\
Okounkov {\cite[Section~4]{Okounkov:1998aa}}) is the inverse limit
\[
\Lamsl=\varprojlim_N\Lamsl(N)
=\bigl\{\,(f_N)_{N\ge1}\ :\ f_{N+1}|_{x_{N+1}=1}=f_N,\ \sup_N\deg f_N<\infty\,\bigr\},
\]
in parallel with \eqref{def:shifted-symm}.  Here the stabilizing value is $x_{N+1}=1$,
the coordinate of an empty row of $q^\lambda$.
We let $\Lamsf=\Lamsl\otimes\QQ(q,t)$ be the $(q,t)$-shifted symmetric functions
with rational coefficients.  We use the same convention on the unshifted side:
$\Laml$ and $\Lamf$ denote ordinary symmetric functions $\Lambda$ with Laurent,
respectively rational, coefficients.
\end{definition}

On $\Lamf$ we use Macdonald's scalar product \cite[Chapter~VI]{Macdonald:1995},
\[
\langle p_\rho,p_\sigma\rangle_{q,t}
=\delta_{\rho\sigma}\,z_\rho\prod_i\frac{1-q^{\rho_i}}{1-t^{\rho_i}},
\]
and plethystic notation.  With $X=x_1+x_2+\cdots$ and $E$ an expression in the
$x_i$, $q$ and $t$, $p_k[E]$ is obtained from $E$ by raising each of the $x_i$,
$q$ and $t$ to the $k$-th power, and $f[E]=\sum_\rho c_\rho\prod_i p_{\rho_i}[E]$
for $f=\sum_\rho c_\rho\,p_\rho$, with the coefficients $c_\rho$ unchanged.  For
example, $p_k[X(1-t)]=(1-t^k)\,p_k$ and $p_k[X+\frac1{1-t}]=p_k+(1-t^k)^{-1}$.

For a partition $\lambda$, we write $q^\lambda=(q^{\lambda_1},q^{\lambda_2},\ldots)$
for the corresponding evaluation point, with trailing zero parts giving
coordinates equal to $1$.
Set $Y_i(q^\lambda):=q^{\lambda_i}t^{1-i}$.
We also write
\begin{equation}\label{def:n}
n(\lambda)=\sum_i(i-1)\lambda_i.
\end{equation}

Define the $(q,t)$-hook factors and products by
\[
c_\lambda(s)=1-q^{\arm_\lambda(s)}t^{\leg_\lambda(s)+1},
\qquad
c'_\lambda(s)=1-q^{\arm_\lambda(s)+1}t^{\leg_\lambda(s)},
\]
\[
c_\lambda(q,t)=\prod_{s\in\lambda}c_\lambda(s),
\qquad
c'_\lambda(q,t)=\prod_{s\in\lambda}c'_\lambda(s).
\]
The integral Macdonald function is
$J_\lambda(q,t):=c_\lambda(q,t)P_\lambda(q,t)$, where $P_\lambda$ is monic in
the monomial basis.  Its norm is
$j_\lambda:=\langle J_\lambda,J_\lambda\rangle_{q,t}=c_\lambda c'_\lambda$
\cite[Chapter~VI]{Macdonald:1995}.

For a pivot $\lambda\sim_p\tilde\lambda$ moving a box from row $a$ to row
$b>a$, the shared hook is
\[
h=c_\lambda(p)=c'_{\tilde\lambda}(p)=1-q^At^B,
\qquad A=\lambda_a-\lambda_b-1,\qquad B=b-a.
\]

The ring $\Lamsl$ carries a convenient Laurent basis,
Okounkov's \emph{stable shifted power sums}
\[
p^\circ_k=\sum_i t^{k(1-i)}\bigl(x_i^k-1\bigr) = p_k(Y(x))-p_k(Y(1)),
\qquad
P^\circ_\kappa=\prod_j p^\circ_{\kappa_j}.
\]
These are manifestly Laurent, and the products $\{P^\circ_\kappa\}$ form a basis of $\Lamsl$.

\begin{lemma}\label{lem:multiplicative-coordinate-pivot}
For $\lambda \sim_p \tilde \lambda$, we have
\[
\{Y_i(q^\lambda)\}_{i\ge1}
\equiv
\{Y_i(q^{\tilde\lambda})\}_{i\ge1}
\pmod h
\]
as multisets, after interchanging the two affected rows $a\leftrightarrow b$.
\end{lemma}
\begin{proof}
Only the two affected rows $(a,b)$ change.  We have
\[
\frac{Y_a(q^\lambda)}{Y_b(q^{\tilde\lambda})} = \frac{q^{\lambda_a}t^{1-a}}{q^{\tilde\lambda_b}t^{1-b}}
=\frac{q^{\lambda_a}t^{1-a}}{q^{\lambda_b+1}t^{1-b}}
= 1-h \equiv 1 \pmod{h},
\]
and similarly
\[
\frac{Y_a(q^{\tilde\lambda})}{Y_b(q^\lambda)}
=\frac{q^{\lambda_a-1}t^{1-a}}{q^{\lambda_b}t^{1-b}}
=1-h \equiv 1 \pmod h.
\]
Thus the two changed shifted coordinates agree after swapping rows modulo
$h$, while all other shifted coordinates are unchanged.
\end{proof}

\begin{lemma}[pivot congruence]\label{lem:macdonald-universal}
Let $\lambda\sim_p\tilde\lambda$ be a pivot with shared hook $h$.
For every $F\in\Lamsl$,
\[
F(q^\lambda)\equiv F(q^{\tilde\lambda})\pmod h .
\]
\end{lemma}
\begin{proof}
By Lemma~\ref{lem:multiplicative-coordinate-pivot}, after interchanging
$a\leftrightarrow b$ the affected $Y$-coordinates differ by multiples of $h$:
multiplication by $1-h$ is congruent to the identity modulo $h$.
The polynomial congruence \eqref{eq:modprop} and symmetry in the $Y_i$ now
apply exactly as in Lemma~\ref{lem:regular-pivot}.
Since $B>0$, the two Laurent monomials in $h=1-q^At^B$ are distinct, and $h$ is
therefore not a unit of $\QQ[q^{\pm},t^{\pm}]$, so the congruence is nonvacuous.
\end{proof}

Shifted Macdonald polynomials, also called interpolation Macdonald
polynomials, were introduced in \cite{Knop:1997ac,Sahi:1996aa,Okounkov:1998aa}.
For each partition $\mu$, there is a unique $(q,t)$-shifted symmetric
function $\Js_\mu(x;q,t)$ of degree at most $|\mu|$ such that
\begin{enumerate}
    
   \item (Vanishing) For $\lambda\ne\mu$ with $|\lambda|\leq|\mu|$,
$$\Js_\mu(q^\lambda)=0$$
   (equivalently, by the extra-vanishing property \cite[Prop.~4.6]{Okounkov:1998aa}, $\Js_\mu(q^\lambda)=0$ unless $\mu\subseteq\lambda$).
\item (Integral Normalization) (cf.~\eqref{def:n})
    $$\Js_\mu(q^\mu)=(-1)^{|\mu|}q^{n(\mu')}t^{-2n(\mu)}\normJ_\mu.$$
\end{enumerate}
Moreover, the top homogeneous part of $\Js_\mu$ is $\J_\mu(Y(x))$
\cite[Theorem~4.1]{Dali:2026aa} (see also \cite{Sahi:1996aa,Okounkov:1998aa}), and
$\Js_\mu\in\Lamsl$ by Knop's integrality theorem (Theorem~\ref{thm:mac-laurent}
below).

The Macdonald Littlewood--Richardson constants $c^\lambda_{\mu\nu}\in\QQ(q,t)
$ are defined by
\[
J_\mu J_\nu=\sum_\lambda c^\lambda_{\mu\nu}(q,t)J_\lambda. \]
We consider their generalization to the \emph{shifted} Macdonald Littlewood--Richardson coefficients defined as
\[
\Js_\mu\Js_\nu=\sum_\lambda d^\lambda_{\mu\nu}(q,t)\,\Js_\lambda,
\qquad d^\lambda_{\mu\nu}\in\QQ(q,t),
\]
supported on $|\lambda|\le|\mu|+|\nu|$.  Exactly as in the Jack case, where
$g^\lambda_{\mu\nu}:=H_\lambda H'_\lambda\,d^\lambda_{\mu\nu}$, the normalized
coefficient appearing in the congruence is defined to be
\[
g_{\mu\nu}^{\lambda}(q,t):=\langle J_\lambda,J_\lambda\rangle_{q,t}\,d^\lambda_{\mu\nu}(q,t)
=c_\lambda c'_\lambda\,d^\lambda_{\mu\nu}(q,t).
\]
In top degree $|\lambda|=|\mu|+|\nu|$, comparison of homogeneous components
gives $d^\lambda_{\mu\nu}(q,t)=c^\lambda_{\mu\nu}(q,t)$, and hence
\[
g^\lambda_{\mu\nu}(q,t)
=c_\lambda c'_\lambda\,c^\lambda_{\mu\nu}(q,t)
=\langle J_\mu J_\nu,J_\lambda\rangle_{q,t}.
\]

\subsection{New features in the Macdonald case}

Two features distinguish the Macdonald case from the Jack case.  The first is
the diagonal evaluation factor $u_\lambda$.  In the integral normalization above,
\[
\Js_\lambda(q^\lambda;q,t)
=u_\lambda\,c_\lambda c'_\lambda,\qquad
u_\lambda=(-1)^{|\lambda|}t^{-2n(\lambda)}q^{n(\lambda')}.
\]
Unlike the Jack diagonal value, $u_\lambda$
is a genuine $q,t$-monomial prefactor that must be carried through every pivot
congruence. Its behaviour under a pivot is controlled by the following lemma.

The second, and more substantial, feature is that the power-sum norms
$\langle p_\rho,p_\rho\rangle_{q,t}=z_\rho\prod_i(1-q^{\rho_i})/(1-t^{\rho_i})$
are not units of $\QQ[q^{\pm},t^{\pm}]$, whereas their Jack analogues
$z_\rho\al^{\ell(\rho)}$ are units of $\QQ[\al^\pm]$.  The Jack proofs of
Proposition~\ref{prop:jack-characters-laurent} and Lemma~\ref{lem:skew-laurent}
use the Jack pairing in the power-sum basis, so they do not transfer directly.
For the shift map we use Knop's inversion formula instead
(Section~\ref{sec:mac-shift}), and for the skew statement a kernel formula for
the shift map (Section~\ref{sec:mac-input}).

\begin{lemma}\label{lem:macdonald-n-shift}
For a pivot $\lambda\sim_p\tilde\lambda$ with hook $h=1-q^A t^B$, one has
\[
n(\tilde\lambda)-n(\lambda)=B,\qquad
n(\tilde\lambda')-n(\lambda')=-A.
\]
\end{lemma}
\begin{proof}
The pivot removes one box from row $a$ and adds one box to row $b>a$.  Hence
\[
n(\tilde\lambda)-n(\lambda)=-(a-1)+(b-1)=b-a=B.
\]
Summing $n(\lambda)=\sum_i(i-1)\lambda_i$ column-by-column gives
$n(\lambda')=\sum_i \binom{\lambda_i}{2}$.  Only
rows $a$ and $b$ change, so
\[
n(\tilde\lambda')-n(\lambda')
=\binom{\lambda_a-1}{2}-\binom{\lambda_a}{2}
 +\binom{\lambda_b+1}{2}-\binom{\lambda_b}{2}
=-\lambda_a+1+\lambda_b=-A.
\]
Here $A=\lambda_a-\lambda_b-1$ and $B=b-a$ are the pivot data fixed above.
\end{proof}

Consequently
\begin{equation}\label{eq:u-ratio}
\frac{u_{\tilde\lambda}}{u_\lambda}
=q^{-A}t^{-2B}
\equiv t^{-B}
=\frac{t^{-n(\tilde\lambda)}}{t^{-n(\lambda)}}
\pmod h,
\end{equation}
so multiplying the congruence
$u_\lambda g^\lambda_{\mu\nu}\equiv
u_{\tilde\lambda}g^{\tilde\lambda}_{\mu\nu}$ of
Proposition~\ref{prop:mac-top-output-pivots} by the unit
$t^{-n(\lambda)}u_\lambda^{-1}$ gives
\[
t^{-n(\lambda)}g^\lambda_{\mu\nu}
\equiv t^{-n(\lambda)}\frac{u_{\tilde\lambda}}{u_\lambda}
       g^{\tilde\lambda}_{\mu\nu}
\equiv t^{-n(\tilde\lambda)}g^{\tilde\lambda}_{\mu\nu}\pmod h.
\]
The same calculation applies to either input partition.  Multiplication by
the fixed factors from the other two partitions then yields the normalization
$t^{-n(\sigma)}$ of Theorem~\ref{thm:mac-pivots-intro}.

\subsection{The shift map}\label{sec:mac-shift}

The $(q,t)$-analogue of Lassalle's shift map appears implicitly in Lassalle's
generalized Macdonald binomial coefficients \cite{Lassalle:1998,Lassalle:1999aa},
and explicitly, in finitely many variables, in Knop's inversion formula
\cite[Corollary~5.2]{Knop:1997ac}, and Ben~Dali--D'Adderio \cite{Dali:2026aa} give a
formula for it in the stable setting (Proposition~\ref{prop:mac-kernel}).  The assignment
\[
\operatorname{sh}_{q,t}(J_\lambda)=\Js_\lambda
\]
is a bijection from the Macdonald basis of $\Lamf$ onto the shifted Macdonald
basis of $\Lamsf$, and hence extends linearly to a $\QQ(q,t)$-linear isomorphism
$\sh_{q,t}\colon\Lamf\xrightarrow{\sim}\Lamsf$.
The following stable formulation is a consequence of Knop's inversion formula
and the Laurent-integrality of its operators
\cite[Corollary~5.2 and the proof of Proposition~5.3]{Knop:1997ac}.
We include the passage to the stable lattice and the unitriangularity argument
for the inverse.

\begin{theorem}[Laurent shift isomorphism]\label{thm:mac-laurent}
The shift map restricts to an isomorphism of $\QQ[q^{\pm},t^{\pm}]$-lattices
$\sh_{q,t}\colon\Laml\xrightarrow{\sim}\Lamsl$.
\end{theorem}
\begin{proof}
Fix a degree $d$ and $N\ge d$, and write $z_i=Y_i$, $i\le N$.  The partition
evaluation points $Y(q^\mu)=q^\mu\rho$, with
$\rho=(1,t^{-1},\dots,t^{1-N})$, are those of \cite{Knop:1997ac}.
Write $\sh_{q,t}^{(N)}$ for Knop's finite-variable map, denoted $\Psi$ in his
paper, restricted to symmetric polynomials.  It sends a homogeneous
symmetric polynomial $f$ of degree $d$ to the unique symmetric polynomial of
degree at most $d$ with top homogeneous component $f$ that vanishes at
$q^\mu\rho$ for every partition $\mu$ with $|\mu|<d$.  His inversion formula
becomes $\sh_{q,t}^{(N)}(f)=q^{-\binom d2}f(Z_1,\dots,Z_N)(1)$, where the
commuting operators $Z_i$ are defined over $\ZZ[q^{\pm},t^{\pm}]$
\cite[Corollary~5.2 and the proof of Proposition~5.3]{Knop:1997ac}.
Hence $\sh_{q,t}^{(N)}$ preserves coefficients in $\QQ[q^{\pm},t^{\pm}]$.
For a homogeneous symmetric function $f$ of degree $d$, the restriction of
$\sh_{q,t}(f)$ to $N$ variables has top homogeneous component $f(Y)$ and
vanishes at every $q^\mu$ with $|\mu|<d$, since $\sh_{q,t}(f)$ is a combination
of $\Js_\lambda$ with $|\lambda|=d$.  Uniqueness within the symmetric
polynomials in $Y$ identifies this restriction with
$\sh_{q,t}^{(N)}(f(Y_1,\dots,Y_N))$, so it has coefficients in
$\QQ[q^{\pm},t^{\pm}]$ when $f\in\Laml$.

To pass to $\Lamsl$, note that in these $N$ variables
$p^\circ_k=p_k(Y)-\sum_{i\le N}t^{k(1-i)}$, so the functions $P^\circ_\kappa$
with $|\kappa|\le d$ restrict to $p_\kappa(Y)$ plus terms of lower degree with
coefficients in $\QQ[t^{-1}]$.  They are linearly independent, and the transition
matrix to the monomials $m_\kappa(Y)$, $|\kappa|\le d$, is block triangular with
respect to degree, with entries in $\QQ[t^{\pm1}]$ and diagonal blocks
invertible over $\QQ$.  Hence the coefficients of $\sh_{q,t}(f)$ in the basis
$\{P^\circ_\kappa\}$ lie in $\QQ[q^{\pm},t^{\pm}]$, that is, $\sh_{q,t}(\Laml)\subseteq\Lamsl$.
Finally, $\sh_{q,t}(p_\mu)$ and $P^\circ_\mu$ both have top homogeneous
component $p_\mu(Y)$, so the matrix of $\sh_{q,t}$ from $\{p_\mu\}$ to
$\{P^\circ_\kappa\}$ is unitriangular with respect to degree, and its inverse
also has entries in $\QQ[q^{\pm},t^{\pm}]$.
\end{proof}

In particular $\Js_\mu=\sh_{q,t}(J_\mu)\in\Lamsl$ for every $\mu$, which is
Knop's integrality theorem for the interpolation Macdonald polynomials
\cite[Proposition~5.3]{Knop:1997ac}.

\subsection{The pivot congruences}\label{sec:mac-pivots}

With Theorem~\ref{thm:mac-laurent} and Proposition~\ref{prop:mac-skew-laurent},
the Jack arguments of Section~\ref{sec:proof} go through over $\QQ[q^{\pm},t^{\pm}]$.

\begin{proposition}[Macdonald output pivots]\label{prop:mac-top-output-pivots}
Let $\lambda\sim_p\tilde\lambda$ be an output pivot with shared hook $h$.  Then
\[
u_{\lambda} g_{\mu\nu}^{\lambda}(q,t)
\equiv
u_{\tilde\lambda} g_{\mu\nu}^{\tilde\lambda}(q,t)
\pmod{h}.
\]
\end{proposition}
\begin{proof}
By Theorem~\ref{thm:mac-laurent}, $\Js_\mu\in\Lamsl$ for every $\mu$, so
\[
\Phi_{\mu\nu}:=\sh_{q,t}^{-1}(\Js_\mu\Js_\nu)
=\sum_\rho d^\rho_{\mu\nu}(q,t)J_\rho
\]
lies in $\Laml$, and so does its homogeneous component
$\Phi^{(k)}_{\mu\nu}$ of degree $k=|\lambda|$.  Hence
\[
F^{(k)}_{\mu\nu}:=\sh_{q,t}(\Phi^{(k)}_{\mu\nu})
=\sum_{\rho\vdash k}d^\rho_{\mu\nu}(q,t)\Js_\rho
\]
lies in $\Lamsl$.  Since every $\rho$ in this sum has
$|\rho|=|\lambda|$, the vanishing property leaves the single term
\[
F^{(k)}_{\mu\nu}(q^\lambda)=d^\lambda_{\mu\nu}(q,t)\,\Js_\lambda(q^\lambda)
=u_\lambda\,g_{\mu\nu}^{\lambda}(q,t),
\]
and Lemma~\ref{lem:macdonald-universal} gives the claim.
\end{proof}

\begin{example}
Take $\mu=\nu=(2,1)$, $\lambda=(3,2,1)$, and
$\tilde\lambda=(2,2,1,1)$.  Then
\[
n(\mu)=n(\nu)=1,\qquad n(\lambda)=4,\qquad n(\tilde\lambda)=7.
\]
Direct computation of the Macdonald coefficients gives
\begin{align*}
t^{-n(\lambda)}g_{\mu\nu}^{\lambda}(q,t)
&=t^{-4}(t-1)^4(q-1)^4(qt^2-1)(q^2t-1)\\
&\quad\times\bigl(2q^5t^5+q^5t^4+q^4t^5-q^5t^3+4q^4t^4
-q^3t^5-q^4t^3-q^3t^4\\
&\qquad-3q^4t^2+4q^3t^3-3q^2t^4-q^4t-q^3t^2
-q^2t^3-qt^4-3q^3t\\
&\qquad+4q^2t^2-3qt^3-q^2t-qt^2-q^2+4qt-t^2+q+t+2\bigr),
\end{align*}
and
\[
t^{-n(\tilde\lambda)}g_{\mu\nu}^{\tilde\lambda}(q,t)
=t^{-7}(t+1)^2(t-1)^4(q-1)^4(q^2t-1)^2(qt^3-1)(qt^4-1).
\]
For the pivot hook
\[
c_{321}(1,1)=c'_{2211}(1,1)=1-q^2t^3,
\]
one checks that
\[
t^{-n(\lambda)}g_{\mu\nu}^{\lambda}(q,t)
\equiv
t^{-n(\tilde\lambda)}g_{\mu\nu}^{\tilde\lambda}(q,t)
\pmod{1-q^2t^3}.
\]
Equivalently, at the specialization $q=s^3,\ t=s^{-2}$ where $h=0$, we find
\begin{align*}
(s^{-2})^{-n(\lambda)}g_{\mu\nu}^{\lambda}(s^3,s^{-2})
&=(s^{-2})^{-n(\tilde\lambda)}g_{\mu\nu}^{\tilde\lambda}(s^3,s^{-2})\\
&=s^{-6}(1-s^2)^2(1-s^4)^4(1-s^3)^5(1-s^5).
\end{align*}
\end{example}

\subsubsection{Input pivots}\label{sec:mac-input}
As in the Jack case, for an input pivot we keep $\lambda,\mu$ fixed and pivot in
$\nu$, and the relevant object is the adjoint of the operator
$L_\mu(f):=\sh_{q,t}^{-1}(\Js_\mu\,\sh_{q,t}(f))$, which satisfies
$L_\mu(J_\tau)=\Phi_{\mu\tau}$.  Define $\Psi_{\lambda/\mu}\in\Lamf$ by the
adjunction
\begin{equation}\label{eq:mac-skew-sum-def}
\langle\Psi_{\lambda/\mu},J_\tau\rangle_{q,t}
=\langle J_\lambda,L_\mu(J_\tau)\rangle_{q,t}
=g^\lambda_{\mu\tau}
\qquad\text{for every }\tau ,
\end{equation}
the last equality from $\langle J_\rho,J_\rho\rangle_{q,t}=c_\rho c'_\rho$ and
$g^\lambda_{\mu\tau}=c_\lambda c'_\lambda\,d^\lambda_{\mu\tau}$.  Equivalently
\[
\Psi_{\lambda/\mu}=L_\mu^\dagger(J_\lambda)
=\sum_\tau d^\lambda_{\mu\tau}(q,t)\frac{c_\lambda c'_\lambda}{c_\tau c'_\tau}\,J_\tau ,
\]
a finite sum, since $d^\lambda_{\mu\tau}=0$ unless $\tau\subseteq\lambda$ by the
extra vanishing property.  The same fact shows that the adjoint $L_\mu^\dagger$
exists: the degree-$n$ part of $L_\mu(J_\tau)$ vanishes for $|\tau|>n$, and
$L_\mu^\dagger(J_\lambda)$ is the finite sum above.

The proof of Lemma~\ref{lem:skew-laurent} does not transfer, since the power-sum
norms are not units.  Instead we use a kernel formula for the shift map due to
Ben Dali--D'Adderio,
together with integrality results of Bergeron--Garsia--Haiman--Tesler.  Let
$\nabla$ be the $\QQ(q,t)$-linear operator on $\Lamf$ of \cite{Dali:2026aa},
defined by
\[
\nabla J_\rho=\beta_\rho\,J_\rho,\qquad
\beta_\rho=(-1)^{|\rho|}q^{n(\rho')}t^{-n(\rho)},
\]
and let $\mathcal T$ be the translation $\mathcal T f=f\bigl[X+\tfrac1{1-t}\bigr]$,
that is, $p_k\mapsto p_k+(1-t^k)^{-1}$.

\begin{proposition}[Ben Dali--D'Adderio]\label{prop:mac-kernel}
For every $f\in\Lamf$ and every partition $\lambda$,
\[
\sh_{q,t}(f)(q^\lambda)
=t^{-n(\lambda)}\bigl\langle\nabla f,\;\mathcal T J_\lambda\bigr\rangle_{q,t}.
\]
\end{proposition}
\begin{proof}
By \cite[Theorem~4.3, (27)]{Dali:2026aa}, whose shift map, normalization of
$\Js_\mu$, evaluation points $q^\lambda$, scalar product and operator $\nabla$
coincide with ours, the left side equals
$t^{-n(\lambda)}\langle\Omega[X/(1-q)]\,\nabla f,J_\lambda\rangle_{q,t}$.  Here
$\Omega[X/(1-q)]=\exp\bigl(\sum_{k\ge1}p_k/(k(1-q^k))\bigr)$ lies in the
completion of $\Lamf$, and the pairing is computed degree by degree.  The
$\langle\,,\rangle_{q,t}$-adjoint of multiplication by $p_k$ is
$k\frac{1-q^k}{1-t^k}\,\partial/\partial p_k$, so
$\langle\Omega[X/(1-q)]\,g,h\rangle_{q,t}=\langle g,\mathcal T h\rangle_{q,t}$ for
$g,h\in\Lamf$, since
$\mathcal T=\exp\bigl(\sum_{k\ge1}(1-t^k)^{-1}\partial/\partial p_k\bigr)$.
\end{proof}

Let
\[
\mathcal A=\{f[X(1-t)]:f\in\Laml\}
\]
be the $\QQ[q^{\pm},t^{\pm}]$-span of the functions $p_\rho[X(1-t)]=\prod_i(1-t^{\rho_i})\,p_\rho$,
so that $\mathcal A\subseteq\Laml$.  By Macdonald \cite[VI~(8.11)]{Macdonald:1995}
and the integrality of the $(q,t)$-Kostka coefficients \cite{Knop:1997ab},
$J_\lambda=\sum_\nu K_{\nu\lambda}(q,t)\,s_\nu[X(1-t)]$ with
$K_{\nu\lambda}\in\ZZ[q,t]$, so $J_\lambda\in\mathcal A$.

\begin{proposition}\label{prop:mac-skew-laurent}
For all $\mu\subseteq\lambda$, $\Psi_{\lambda/\mu}\in\Laml$.
\end{proposition}
\begin{proof}
We\footnote{This proof was greatly assisted by Anthropic Claude Opus 5.5, the use of the Bergeron--Garsia--Haiman--Tesler results is due to it.} show more, namely $\Psi_{\lambda/\mu}\in\mathcal A$.  Since
$\Psi_{\lambda/\mu}=L_\mu^\dagger(J_\lambda)$ and $J_\lambda\in\mathcal A$, it
suffices to show that $L_\mu^\dagger$ preserves $\mathcal A$.
First we identify the adjoint.  We then conjugate it by a semilinear map
$\Theta$ taking $\mathcal A$ onto $\Laml$, and express the conjugate as a
composition of Laurent-preserving operators.  The adjoint formula is
\[
L_\mu^\dagger=\nabla\,\mathcal T\,D_\mu\,\mathcal T^{-1}\nabla^{-1},
\qquad D_\mu J_\tau=\Js_\mu(q^\tau)\,J_\tau .
\]
To prove the formula, let $f\in\Lamf$.  For every partition $\rho$,
Proposition~\ref{prop:mac-kernel} gives
\[
t^{-n(\rho)}\langle\nabla L_\mu f,\mathcal TJ_\rho\rangle_{q,t}
=\sh_{q,t}(L_\mu f)(q^\rho)
=\Js_\mu(q^\rho)\,\sh_{q,t}(f)(q^\rho)
=t^{-n(\rho)}\langle\nabla f,\mathcal T D_\mu J_\rho\rangle_{q,t}.
\]
The operator $\nabla$ is diagonal in the orthogonal basis $\{J_\tau\}$, hence
self-adjoint, so
$\langle L_\mu f,\nabla\mathcal T g\rangle_{q,t}=\langle f,\nabla\mathcal T D_\mu g\rangle_{q,t}$
for all $f,g$.  Taking $g=\mathcal T^{-1}\nabla^{-1}w$ gives the formula for
$L_\mu^\dagger$.

For the integrality we pass to the modified Macdonald functions
$\tilde H_\tau=\sum_\nu t^{n(\tau)}K_{\nu\tau}(q,t^{-1})\,s_\nu$
\cite[Definition~3.5.3 and the remark following it]{Haiman:2003}.
Let $\Theta$ be the ring automorphism of $\Lamf$ given by
\[
\Theta\Bigl(\sum_\kappa c_\kappa(q,t)\,p_\kappa\Bigr)
=\sum_\kappa c_\kappa(q,t^{-1})\prod_i\frac{p_{\kappa_i}}{1-t^{-\kappa_i}} .
\]
Thus $\Theta(cf)=c(q,t^{-1})\,\Theta(f)$ for $c\in\QQ(q,t)$, and
$\Theta\bigl(\prod_i(1-t^{\rho_i})\,p_\rho\bigr)=p_\rho$.  Hence $\Theta$ maps
$\mathcal A$ onto $\Laml$, and $\Theta(s_\nu[X(1-t)])=s_\nu$.  Applying $\Theta$
to $J_\tau=\sum_\nu K_{\nu\tau}(q,t)\,s_\nu[X(1-t)]$ gives
\begin{equation}\label{eq:theta-J}
\Theta(J_\tau)=t^{-n(\tau)}\tilde H_\tau .
\end{equation}
Write $\Js_\mu=\sum_\kappa a_\kappa P^\circ_\kappa$, let $\epsilon$ act by
$(-1)^d$ in degree $d$, and set $\mathcal T_1g=g[X+1]$.  On the modified
Macdonald basis define
\[
\tilde\nabla\tilde H_\tau=t^{n(\tau)}q^{n(\tau')}\tilde H_\tau,
\qquad
\Delta_f\tilde H_\tau=f[B_\tau]\tilde H_\tau,
\qquad B_\tau=\sum_{(i,j)\in\tau}q^{j-1}t^{i-1}.
\]
The target factorization is
\begin{equation}\label{eq:mac-conjugated-adjoint}
\Theta L_\mu^\dagger\Theta^{-1}
=\epsilon\tilde\nabla\mathcal T_1
\Bigl(\sum_\kappa a_\kappa(q,t^{-1})
\prod_j(q^{\kappa_j}-1)\Delta_{p_\kappa}\Bigr)
\mathcal T_1^{-1}\tilde\nabla^{-1}\epsilon.
\end{equation}
The calculations (a)--(c) below establish this identity.  We then apply the
operator-integrality theorem of Bergeron--Garsia--Haiman--Tesler
\cite{BGHT:1999} to the $\tilde\nabla^{\pm1}$ and $\Delta_f$ factors.
The precise preservation statement is recalled at the end of the proof.
Since $\Theta$ is semilinear,
$\Theta\mathcal X\Theta^{-1}$ is $\QQ(q,t)$-linear whenever $\mathcal X$ is,
and is a ring automorphism when $\mathcal X$ is.  We may therefore check these
identities on the basis $\{\tilde H_\tau\}$ or, for ring automorphisms, on the
power sums.

\emph{(a) The operator $\nabla$.}  By \eqref{eq:theta-J} and semilinearity,
\[
\Theta\nabla\Theta^{-1}(\tilde H_\tau)
=t^{n(\tau)}\,\Theta(\beta_\tau J_\tau)
=\beta_\tau(q,t^{-1})\,\tilde H_\tau
=(-1)^{|\tau|}t^{n(\tau)}q^{n(\tau')}\,\tilde H_\tau .
\]
Thus $\Theta\nabla\Theta^{-1}=\epsilon\tilde\nabla$, with
$\tilde\nabla$ the operator $\nabla$ of \cite{BGHT:1999}.

\emph{(b) The operator $D_\mu$.}  The coefficients $a_\kappa$ lie in $\QQ[q^{\pm},t^{\pm}]$ by
Theorem~\ref{thm:mac-laurent}.  Since
$q^{k\tau_i}-1=(q^k-1)\sum_{j=1}^{\tau_i}q^{k(j-1)}$,
\[
p^\circ_k(q^\tau)=\sum_i t^{k(1-i)}\bigl(q^{k\tau_i}-1\bigr)
=(q^k-1)\sum_{(i,j)\in\tau}q^{k(j-1)}t^{k(1-i)} .
\]
Replacing $t$ by $t^{-1}$ turns this into $(q^k-1)\,p_k[B_\tau]$.
By \eqref{eq:theta-J},
$\Theta D_\mu\Theta^{-1}$ therefore multiplies $\tilde H_\tau$ by
$\sum_\kappa a_\kappa(q,t^{-1})\prod_j(q^{\kappa_j}-1)\,p_\kappa[B_\tau]$, that is,
\[
\Theta D_\mu\Theta^{-1}
=\sum_\kappa a_\kappa(q,t^{-1})\prod_j(q^{\kappa_j}-1)\,\Delta_{p_\kappa},
\]
where the $\Delta_f$ are the operators of \cite{BGHT:1999}.

\emph{(c) The translation $\mathcal T$.}  $\Theta\mathcal T\Theta^{-1}$ and
$\mathcal T_1$ are both $\QQ(q,t)$-linear ring
automorphisms, and they agree on power sums:
$\Theta\mathcal T\Theta^{-1}(p_k)=\Theta\bigl(\mathcal T((1-t^k)p_k)\bigr)
=\Theta\bigl((1-t^k)p_k+1\bigr)=p_k+1$.  Hence
$\Theta\mathcal T\Theta^{-1}=\mathcal T_1$.

Combining (a)--(c) proves \eqref{eq:mac-conjugated-adjoint}.
By \cite{BGHT:1999} (see \cite[Proposition~3.5.14]{Haiman:2003}), $\tilde\nabla$ and
$\Delta_f$, for $f$ with integer coefficients, preserve symmetric functions with
coefficients in $\ZZ[q,t]$, and $\tilde\nabla^{-1}$ preserves those with
coefficients in $\ZZ[q^{\pm},t^{\pm}]$.  Every element of $\Laml$ has the
form $c\,q^{-a}t^{-b}F$ with $c\in\QQ$ and $F$ having coefficients in
$\ZZ[q,t]$, so by $\QQ(q,t)$-linearity all three operators preserve $\Laml$.  The
power sums $p_\kappa$ have integer coefficients and the scalars
$a_\kappa(q,t^{-1})\prod_j(q^{\kappa_j}-1)$ lie in $\QQ[q^{\pm},t^{\pm}]$, so the middle factor
preserves $\Laml$.  So do $\epsilon$ and $\mathcal T_1^{\pm1}$, which send $p_k$
to $p_k\pm1$.  Hence
$\Theta L_\mu^\dagger\Theta^{-1}$ preserves $\Laml$, that is, $L_\mu^\dagger$
preserves $\Theta^{-1}(\Laml)=\mathcal A$.
\end{proof}

\begin{proposition}[Macdonald input pivots]\label{prop:mac-input-pivots}
Let $\nu\sim_p\tilde\nu$ be an input pivot with shared hook $h$, and let
$m=|\nu|=|\tilde\nu|$.  Then
\[
u_\nu\,g_{\mu\nu}^{\lambda}(q,t)
\equiv
u_{\tilde\nu}\,g_{\mu\tilde\nu}^{\lambda}(q,t)
\pmod{h}.
\]
\end{proposition}
\begin{proof}
If $\mu\not\subseteq\lambda$ both sides vanish by the extra vanishing property,
so assume $\mu\subseteq\lambda$.  By Proposition~\ref{prop:mac-skew-laurent},
$\Psi_{\lambda/\mu}\in\Laml$, hence so is its homogeneous component of
degree $m$,
\[
\Psi^{(m)}_{\lambda/\mu}=\sum_{\tau\vdash m}
\frac{g^\lambda_{\mu\tau}(q,t)}{c_\tau c'_\tau}\,J_\tau ,
\]
and $\sh_{q,t}(\Psi^{(m)}_{\lambda/\mu})\in\Lamsl$ by
Theorem~\ref{thm:mac-laurent}.  For $\nu\vdash m$ the vanishing property
leaves only the term $\tau=\nu$, and $\Js_\nu(q^\nu)=u_\nu c_\nu c'_\nu$, so
\[
\sh_{q,t}(\Psi^{(m)}_{\lambda/\mu})(q^\nu)
=u_\nu\,g^\lambda_{\mu\nu}(q,t).
\]
Lemma~\ref{lem:macdonald-universal}, applied at the pivot
$\nu\sim_p\tilde\nu$, gives the claim.
\end{proof}

\begin{example}
Take $\lambda=(3,3,1)$, $\mu=(2,1)$ and the input pivot
$\nu=(3,1)\sim_p\tilde\nu=(2,2)$, with $n(\nu)=1$, $n(\tilde\nu)=2$ and pivot
hook $h=1-qt$.  Direct computation gives
\begin{align*}
g_{\mu\nu}^{\lambda}(q,t)
&=(t+1)(q+1)^2(t-1)^4(q-1)^5(qt^2-1)^2(q^3t-1)^2(q^2t^3-1),\\
g_{\mu\tilde\nu}^{\lambda}(q,t)
&=(t+1)^2(q+1)^2(q-1)^4(t-1)^5(qt^2-1)(q^2t-1)^2(q^3t-1)(q^2t^3-1),
\end{align*}
and
\[
t^{-1}g^{\lambda}_{\mu\nu}(q,t)-t^{-2}g^{\lambda}_{\mu\tilde\nu}(q,t)=hf,
\]
where
\begin{align*}
f
&=-t^{-2}(t+1)(q+1)^2(q-1)^4(t-1)^4(qt^2-1)(q^3t-1)(q^2t^3-1)\\
&\quad\times\bigl(q^4t^3-2q^3t^3+q^3t^2-2q^2t^2+2q^2t
+qt^2-2qt+t^2+t-1\bigr).
\end{align*}
This is the twisted congruence of Theorem~\ref{thm:mac-pivots-intro} for this pivot,
after cancelling the common unit $t^{-n(\mu)-n(\lambda)}$.
\end{example}

\subsubsection{Proof of the Macdonald congruences}\label{sec:mac-main}

\begin{proof}[Proof of Theorem~\ref{thm:mac-pivots-intro}]
By the proof of Proposition~\ref{prop:mac-top-output-pivots},
$u_\lambda g^\lambda_{\mu\nu}=F^{(k)}_{\mu\nu}(q^\lambda)$, where $k=|\lambda|$
and $F^{(k)}_{\mu\nu}\in\Lamsl$.  Its value at $q^\lambda$ lies in $\QQ[q^{\pm},t^{\pm}]$ and
$u_\lambda$ is a unit, proving that $g^\lambda_{\mu\nu}(q,t)\in\QQ[q^{\pm},t^{\pm}]$.

Propositions~\ref{prop:mac-top-output-pivots} and~\ref{prop:mac-input-pivots}
give the congruences with the factors $u$ for pivots in $\lambda$ and in $\nu$,
and pivots in $\mu$ follow from $g^\lambda_{\mu\nu}=g^\lambda_{\nu\mu}$.  By
\eqref{eq:u-ratio} and its analogue for $\nu$, and since the remaining factors
$t^{-n(\cdot)}$ are units, each congruence is equivalent to
$t^{-n(\sigma)}g^\lambda_{\mu\nu}\equiv t^{-n(\tilde\sigma)}g^{\tilde\lambda}_{\tilde\mu\tilde\nu}$
modulo $h$.  This proves Theorem~\ref{thm:mac-pivots-intro}.
\end{proof}

\section{GKM and the Hilbert scheme of points}\label{sec:gkm}

This section is independent of the proofs above.  Until
Section~\ref{sec:ktheory} we return to the Jack notation and interpret its
pivot congruences in the torus-equivariant cohomology of
$\Hilb^n(\BC^2)$.  The torus-fixed points are indexed by partitions of $n$,
and their tangent weights are, up to sign, the Jack hooks.
Every Laurent shifted symmetric function has a tautological cohomology-class
lift (Proposition~\ref{prop:tautological}), and pivot-related fixed points are
joined by explicit $T$-invariant curves with the shared hook as tangent weight.
The Jack congruences are the necessary Goresky--Kottwitz--MacPherson (GKM)
edge conditions along these curves, but they do not give a complete description
of the image of restriction.  We follow the conventions of
Okounkov--Pandharipande \cite{Okounkov:2010a}.
The final subsection gives the corresponding Macdonald interpretation in
rational equivariant $K$-theory.

\subsection{The Hilbert scheme and its torus action}
Let $\Hilb^n(\BC^2)$ be the Hilbert scheme parametrizing ideals
$I\subseteq\BC[x,y]$ of colength $\dim_\BC\BC[x,y]/I=n$.  It is a smooth
quasi-projective variety of dimension $2n$.  The torus $T=(\BC^*)^2$ acts by
scaling the coordinates, $(s_1,s_2)\cdot(x,y)=(s_1x,s_2y)$, with equivariant
parameters $(\epsilon_1,\epsilon_2)$, so that
$\HT(\mathrm{pt})=\QQ[\epsilon_1,\epsilon_2]$ and the tangent space to $\BC^2$
at the origin has weights $\epsilon_1,\epsilon_2$.  The induced action on
functions is contragredient: the monomial $x^{j-1}y^{i-1}$ has weight
$-(j-1)\epsilon_1-(i-1)\epsilon_2$.  We use this convention for the
fibres of the tautological bundle $\mathcal O/I$.
The $T$-fixed points are the monomial ideals
\[
  I_\lambda=\bigl(x^{j-1}y^{i-1}\ :\ (i,j)\notin\lambda\bigr),\qquad\lambda\vdash n,
\]
so that the box $(i,j)\in\lambda$ corresponds to the monomial $x^{j-1}y^{i-1}$
of $\BC[x,y]/I_\lambda$, with arms pointing in the $x$-direction and legs in
the $y$-direction \cite[(12)]{Okounkov:2010a}.  The fixed points are isolated.

We identify the Jack parameter with the equivariant ratio
\[
  \al=-\epsilon_1/\epsilon_2 ,
\]
and work on the slice $\epsilon_2=1$, $\epsilon_1=-\al$, that is, we base change
along $\HT(\mathrm{pt})\to\QQ[\al]$, and we write $\HT(\Hilb^n)[\al^{-1}]$ for the
resulting module over $\QQ[\al^{\pm}]$.  (Okounkov--Pandharipande
\cite[\S1.1]{Okounkov:2010a} identify $-t_2/t_1$ with the Calogero--Sutherland
coupling $\theta$, and comparison of their operator \cite[(10)]{Okounkov:2010a}
with the Laplace--Beltrami operator of \cite{Stanley:1989} shows that its
eigenfunctions are Jack functions with parameter $1/\theta=\al$.)

\subsection{Tangent weights are hooks}
The tangent space at $I_\lambda$ is
$\operatorname{Hom}_{\BC[x,y]}(I_\lambda,\BC[x,y]/I_\lambda)$.  In the
conventions above its character is
\[
  T_{I_\lambda}\Hilb^n
  =\sum_{s\in\lambda}\Bigl(t_1^{\arm(s)+1}t_2^{-\leg(s)}
     +t_1^{-\arm(s)}t_2^{\leg(s)+1}\Bigr),\qquad t_i=e^{\epsilon_i},
\]
which is \cite[Prop.~5.8]{Nakajima:1999a} (see also \cite{Nakajima:1996b,Li:2004a}) written in
the present conventions.  Written additively, with $\epsilon_1=-\al\epsilon_2$,
the two tangent weights at a box $s\in\lambda$ are
\[
  t_1^{\arm(s)+1}t_2^{-\leg(s)}\ \longmapsto\ -\epsilon_2\,h'_\lambda(s),
  \qquad
  t_1^{-\arm(s)}t_2^{\leg(s)+1}\ \longmapsto\ \epsilon_2\,h_\lambda(s),
\]
so on the slice $\epsilon_2=1$ the $2n$ tangent weights are the lower hooks
$h_\lambda(s)$ and the negated upper hooks $-h'_\lambda(s)$ that enter the Jack
norm.  Their product is the tangent Euler class,
\begin{equation}\label{eq:euler-norm}
  e\bigl(T_{I_\lambda}\Hilb^n\bigr)=(-1)^nH_\lambda H'_\lambda
  =(-1)^n\langle \J_\lambda,\J_\lambda\rangle_\al .
\end{equation}
In particular no tangent weight vanishes.

\subsection{Pivot curves}
Pivot-related fixed points are joined by $T$-invariant curves.  Let
$\lambda\sim_p\tilde\lambda$ move a box from row $a$ to row $b>a$, and let
$m=x^{\lambda_a-1}y^{a-1}$ and $\tilde m=x^{\lambda_b}y^{b-1}$ be the monomials
of the removed and the added box.  The union $\lambda\cup\tilde\lambda$ is a
diagram with $n+1$ boxes, and $xm$, $ym$, $x\tilde m$, $y\tilde m$ all lie
outside it, so the ideals
\[
  J_{[w_0:w_1]}=I_{\lambda\cup\tilde\lambda}+\bigl(w_0\,m+w_1\,\tilde m\bigr),
  \qquad [w_0:w_1]\in\mathbb P^1,
\]
all have colength $n$.  This is a $T$-invariant $\mathbb P^1\subset\Hilb^n$ with
$J_{[1:0]}=I_{\tilde\lambda}$ and $J_{[0:1]}=I_\lambda$.  Its tangent weight at
$I_\lambda$ is the weight of the deformation $\tilde m\mapsto m$,
\[
  -(\lambda_a-\lambda_b-1)\,\epsilon_1+(b-a)\,\epsilon_2=\epsilon_2\,h ,
\]
where $h=h_\lambda(p)=h'_{\tilde\lambda}(p)$ is the shared hook of
\eqref{eq:shared-hook}, and at $I_{\tilde\lambda}$ the weight is
$-\epsilon_2 h$.  Thus the two halves of \eqref{eq:shared-hook} are the two ends
of one curve: the lower hook of $\lambda$ and the (negated) upper hook of
$\tilde\lambda$.  For $\lambda=(n)$, $\tilde\lambda=(n-1,1)$ this is the line
$[w_0:w_1]\mapsto(w_0x^{n-1}+w_1y,\ x^n,\ xy,\ y^2)$ of
\cite[\S3.9.2]{Okounkov:2010a}.

\subsection{The Nakajima--Fock identification}
The theorem of Nakajima and Grojnowski
\cite{Nakajima:1999a,Grojnowski:1996a}, in its $T$-equivariant form
\cite[(3)]{Okounkov:2010a} (see \cite{Nakajima:1996b,Vasserot:2001a,Li:2004a}
for Jack polynomials and one-dimensional tori), gives an isomorphism of graded
modules, which after passing to the slice
and to the fraction field reads
\[
  \bigoplus_{n\ge0}\HT\bigl(\Hilb^n(\BC^2)\bigr)\otimes\QQ(\al)\;\cong\;\Lamfa ,
\]
under which $\HT(\Hilb^n)$ is the degree-$n$ part (free of rank $p(n)$) and the
fixed-point class $[I_\lambda]$ corresponds to the integral Jack function
$\J_\lambda$, up to the change of variables of \cite[(11)]{Okounkov:2010a} and a
factor depending only on $n$ \cite[\S2.4]{Okounkov:2010a}.  The product of
$\Lamfa$ transported by this isomorphism is the product defined by Nakajima's
creation operators, rather than the cup product, under which distinct fixed-point
classes are orthogonal.  Since
$[I_\rho]|_{I_\lambda}=0$ for $\rho\ne\lambda$, while
$[I_\lambda]|_{I_\lambda}=e(T_{I_\lambda}\Hilb^n)=(-1)^nH_\lambda H'_\lambda$ by
\eqref{eq:euler-norm}, the interpolation conditions
\eqref{eq:shifted-jack-interpolation} say that for $\rho\vdash n$
\[
  \Js_\rho(\lambda)=(-1)^n\al^{-n}\,[I_\rho]\big|_{I_\lambda}
  \qquad(\lambda\vdash n).
\]
On the partitions of size $n$, the shifted Jack functions are therefore the
fixed-point classes of $\Hilb^n$, seen through their restrictions.

\subsection{Tautological classes}
Evaluation at the partitions of size $n$ turns a shifted symmetric function into
a function on the fixed points of $\Hilb^n$.  Every Laurent one arises by
restriction from an equivariant class.

\begin{proposition}\label{prop:tautological}
For every $F\in\LamsA$ and $n\ge0$ there is a class
$F_n\in\HT(\Hilb^n)[\al^{-1}]$, a polynomial with coefficients in
$\QQ[\al^{\pm}]$ in the Chern characters of the tautological bundle
$\mathcal O/I$, such that
\[
  F_n\big|_{I_\lambda}=F(\lambda)\qquad(\lambda\vdash n).
\]
\end{proposition}
\begin{proof}
The fibre of $\mathcal O/I$ at $I_\lambda$ is $\BC[x,y]/I_\lambda$, spanned by the
monomials $x^{j-1}y^{i-1}$, $(i,j)\in\lambda$, whose weights on the slice are
$\al(j-1)-(i-1)$.  Summing row by row and using $\al\lambda_i=z_i+i$,
\[
  \operatorname{ch}_u(\mathcal O/I)\big|_{I_\lambda}
  =\sum_{i\ge1}\ \sum_{j=1}^{\lambda_i}e^{u(\al(j-1)-(i-1))}
  =\frac{e^u}{e^{u\al}-1}\sum_{i\ge1}\bigl(e^{uz_i}-e^{-ui}\bigr)
  =\frac{e^u}{e^{u\al}-1}\sum_{k\ge1}\frac{u^k}{k!}\,p^\circ_k(\lambda),
\]
where $\operatorname{ch}_u=\sum_{j\ge0}u^j\operatorname{ch}_j$.  Hence
$p^\circ_k(\lambda)$ is the restriction of
$k!\,[u^k]\bigl(e^{-u}(e^{u\al}-1)\operatorname{ch}_u(\mathcal O/I)\bigr)$, a
$\QQ[\al]$-linear combination of $\operatorname{ch}_0,\dots,
\operatorname{ch}_{k-1}$.  Since restriction to fixed points is a ring
homomorphism and the $P_\nu=\prod_jp^\circ_{\nu_j}$ form a
$\QQ[\al^{\pm}]$-basis of $\LamsA$, the claim follows.
\end{proof}

\subsection{The congruences as GKM edge conditions}
We use two formal properties of equivariant cohomology.  First, $\Hilb^n$ has no
odd cohomology, so $\HT(\Hilb^n)$ is a free $\HT(\mathrm{pt})$-module and by the
localization theorem restriction to the fixed points is injective,
\[
  \HT(\Hilb^n)[\al^{-1}]\hookrightarrow\bigoplus_{\lambda\vdash n}\QQ[\al^{\pm}],
  \qquad F\mapsto\bigl(F|_{I_\lambda}\bigr)_\lambda,
\]
and injectivity survives the base change to the slice, since the fixed-point classes
still restrict to a diagonal matrix with the nonzero entries
\eqref{eq:euler-norm}.  Second, if a $T$-invariant $\mathbb P^1$ joins fixed
points $F,F'$ with tangent weight $\chi$ at $F$, restricting to it and using
$\HT(\mathbb P^1)=\{(f,f'):f\equiv f'\pmod\chi\}$ shows that every tuple in the
image satisfies
\begin{equation}\label{eq:gkm-edge}
  f_F\equiv f_{F'}\pmod{\chi}.
\end{equation}
These are the GKM edge conditions \cite{Goresky:1998a,Chang:1974a} (see
\cite{Tymoczko:2005a} for an exposition).  We use them only as necessary
conditions along the pivot curves, without claiming that they characterize
the image of restriction.  Along these curves, \eqref{eq:gkm-edge} reads
$f_\lambda\equiv f_{\tilde\lambda}\pmod h$.

Combined with Proposition~\ref{prop:tautological}, this recovers
Lemma~\ref{lem:regular-pivot}: for $F\in\LamsA$ the tuple
$(F(\lambda))_{\lambda\vdash n}$ is the restriction of the class $F_n$, so it
satisfies \eqref{eq:gkm-edge} along every pivot curve.  On the combinatorial side
the same mechanism is Lemma~\ref{lem:shifted-coordinate-pivot}: writing $s_{ab}$
for the transposition of the $a$-th and $b$-th coordinates,
\[
  z(\lambda)-s_{ab}\,z(\tilde\lambda)=h\,(e_a-e_b),
\]
so a function symmetric in the $z_i$ takes congruent values at $\lambda$ and
$\tilde\lambda$ modulo $h$ as soon as its coefficients are regular along $h=0$.
In either form the Laurent property is the essential input.

\subsection{The class \texorpdfstring{$G_{\mu\nu}$}{G(mu,nu)}}
Fix $\mu,\nu$ and an output size $n\le|\mu|+|\nu|$.  By the proof of
Proposition~\ref{prop:output-pivots},
$g^\lambda_{\mu\nu}=\al^n\,\sh(\Phi^{(n)}_{\mu\nu})(\lambda)$ for $\lambda\vdash n$,
with $\sh(\Phi^{(n)}_{\mu\nu})\in\LamsA$.  Proposition~\ref{prop:tautological}
therefore gives a class
\[
  G_{\mu\nu}\in\HT\bigl(\Hilb^n\bigr)[\al^{-1}],\qquad
  G_{\mu\nu}\big|_{I_\lambda}=g^\lambda_{\mu\nu}\quad(\lambda\vdash n),
\]
unique by injectivity of restriction, and
Proposition~\ref{prop:output-pivots} is the edge condition \eqref{eq:gkm-edge}
for $G_{\mu\nu}$ along the pivot curves.

Expanding in the fixed-point classes, which is possible after inverting the
Euler classes, gives
\[
  G_{\mu\nu}=\sum_{\lambda\vdash n}
  \frac{g^\lambda_{\mu\nu}}{e(T_{I_\lambda}\Hilb^n)}\,[I_\lambda]
  =(-1)^n\sum_{\lambda\vdash n}d^\lambda_{\mu\nu}\,[I_\lambda].
\]
So $d^\lambda_{\mu\nu}$ is, up to sign, the coefficient of $G_{\mu\nu}$ in the
fixed-point basis, and $g^\lambda_{\mu\nu}$ is its restriction to $I_\lambda$.
The two differ by the Euler class $(-1)^nH_\lambda H'_\lambda$.  This is why the
$d^\lambda_{\mu\nu}$ can have hook denominators while the $g^\lambda_{\mu\nu}$
are Laurent: $G_{\mu\nu}$ is a genuine class, but its fixed-point expansion is
not.  In the top degree $n=|\mu|+|\nu|$ we have $d^\lambda_{\mu\nu}=c^\lambda_{\mu\nu}$,
and under the Nakajima--Fock identification $G_{\mu\nu}$ corresponds, up to sign
and the normalization of $[I_\lambda]\leftrightarrow\J_\lambda$, to the product
$\J_\mu\J_\nu$.  For $n<|\mu|+|\nu|$ we do not know a comparably direct
geometric description of $G_{\mu\nu}$.

\subsection{Input pivots}
The class $G_{\mu\nu}$ lives on a single $\Hilb^n$, with $\mu,\nu$ fixed, so its
edge conditions constrain only the output partition.  The input pivots admit the
same treatment.  Fix $\lambda$ and $\mu$.  By the proof of
Proposition~\ref{prop:input-pivots},
$\sh(\Psi^{(m)}_{\lambda/\mu})(\nu)=\al^{-m}g^\lambda_{\mu\nu}$ for $|\nu|=m$,
with $\sh(\Psi^{(m)}_{\lambda/\mu})\in\LamsA$.  Proposition~\ref{prop:tautological}
therefore gives a unique class
$G^{\lambda}_{\mu\,\cdot}\in\HT(\Hilb^{m})[\al^{-1}]$ with
$G^{\lambda}_{\mu\,\cdot}|_{I_\nu}=g^\lambda_{\mu\nu}$ for $\nu\vdash m$, whose
edge conditions along the pivot curves are Proposition~\ref{prop:input-pivots}.

\subsection{A joint class in top degree}
In top degree the three congruence families arise from one class on a product
of Hilbert schemes.  We retain the slice and the coefficient ring
$\QQ[\al^{\pm}]$ used throughout the cohomological discussion.

\begin{proposition}\label{prop:joint-jack-class}
Let $r,s\ge0$ and $n=r+s$.  There is a class
\[
\mathbb M_{r,s}\in
\HT(\Hilb^r\times\Hilb^s\times\Hilb^n)[\al^{-1}]
\]
whose restriction at $(I_\mu,I_\nu,I_\lambda)$, for
$\mu\vdash r$, $\nu\vdash s$, $\lambda\vdash n$, is
$(-1)^n g^\lambda_{\mu\nu}$.  No tangent Euler classes are inverted in its
construction.
\end{proposition}
\begin{proof}
Set $S_\rho=\sh(p_\rho)\in\LamsA$.  By the coefficient identity in the proof
of Corollary~\ref{cor:jack-pivot},
\[
[p_\rho]\J_\mu
=\frac{\al^{r-\ell(\rho)}}{z_\rho}S_\rho(\mu)
\qquad(\rho,\mu\vdash r).
\]
The union index below comes from multiplication in the power-sum basis.
Write $\rho\cup\tau$ for the partition obtained by concatenating and sorting
the parts of $\rho$ and $\tau$, retaining their multiplicities, so for example,
$(2,1)\cup(1)=(2,1,1)$.  This is a multiset union of parts, not the union of
Young diagrams used in the pivot-curve construction.  Since
$p_\rho=\prod_i p_{\rho_i}$, we have
\[
p_\rho p_\tau
=\Bigl(\prod_i p_{\rho_i}\Bigr)\Bigl(\prod_j p_{\tau_j}\Bigr)
=p_{\rho\cup\tau}.
\]
Thus multiplying the power-sum expansions of $\J_\mu$ and $\J_\nu$ gives
\[
g^\lambda_{\mu\nu}
=\sum_{\substack{\rho\vdash r\\\tau\vdash s}}
\bigl([p_\rho]\J_\mu\bigr)\bigl([p_\tau]\J_\nu\bigr)
\langle p_{\rho\cup\tau},\J_\lambda\rangle_\al.
\]
Because $|\rho\cup\tau|=|\lambda|=n$, \eqref{eq:lassalle-shift} gives
$\langle p_{\rho\cup\tau},\J_\lambda\rangle_\al
=\al^n S_{\rho\cup\tau}(\lambda)$.  Substituting this and the coefficient
identity above, and using $r+s=n$, yields
\begin{equation}\label{eq:joint-jack-evaluation}
g^\lambda_{\mu\nu}
=\sum_{\substack{\rho\vdash r\\\tau\vdash s}}
\frac{\al^{2n-\ell(\rho)-\ell(\tau)}}{z_\rho z_\tau}
S_\rho(\mu)S_\tau(\nu)S_{\rho\cup\tau}(\lambda).
\end{equation}
Let $S_{\rho,d}\in\HT(\Hilb^d)[\al^{-1}]$ be the tautological lift of
$S_\rho$ supplied by Proposition~\ref{prop:tautological}.  Then
\[
\mathbb M_{r,s}
:=(-1)^n\sum_{\substack{\rho\vdash r\\\tau\vdash s}}
\frac{\al^{2n-\ell(\rho)-\ell(\tau)}}{z_\rho z_\tau}
S_{\rho,r}\boxtimes S_{\tau,s}\boxtimes S_{\rho\cup\tau,n}
\]
has the required restrictions.  Every coefficient lies in $\QQ[\al^{\pm}]$,
so the construction requires no further localization.
\end{proof}

The edge conditions of $\mathbb M_{r,s}$ along pivot curves in its three
factors give all three top-degree congruence families at once.  With the
fixed-point normalization $[I_\lambda]\leftrightarrow\J_\lambda$, this class
also represents symmetric-function multiplication under the Nakajima--Fock
identification.  Indeed, writing $\pi_i$ for the projections, equivariant
localization gives
\[
\pi_{3*}\bigl(\pi_1^*[I_\mu]\cdot\pi_2^*[I_\nu]\cdot\mathbb M_{r,s}\bigr)
=\sum_{\lambda\vdash n}c^\lambda_{\mu\nu}[I_\lambda],
\]
since $\mathbb M_{r,s}|_{(I_\mu,I_\nu,I_\lambda)}
=c^\lambda_{\mu\nu}e(T_{I_\lambda}\Hilb^n)$.
Here the nonproper pushforward is interpreted after equivariant localization,
while the kernel class itself is defined before inverting Euler classes.
This does not assert the existence of a properly supported correspondence or
of a class over $\QQ[\epsilon_1,\epsilon_2]$ before the slice and base change.

\subsection{Macdonald functions and \texorpdfstring{$K$}{K}-theory}\label{sec:ktheory}
We now return to the Macdonald notation of Section~\ref{sec:macdonald} and
work with the full two-dimensional torus, without taking the cohomological
slice.  We use rational equivariant $K$-theory
$\KT(\Hilb^n)_\QQ:=\KT(\Hilb^n)\otimes_\ZZ\QQ$ and identify its coefficient
ring $R(T)\otimes_\ZZ\QQ$ with $\QQ[q^{\pm},t^{\pm}]$ by
$(q,t)=(e^{\epsilon_1},e^{-\epsilon_2})$.
This is the $K$-theoretic counterpart of the Jack picture anticipated in
\cite{Nakajima:1996b}, and the inversion of $t$ also agrees with Haiman's
normalization
$\tilde H_\mu=t^{n(\mu)}J_\mu[X/(1-t^{-1});q,t^{-1}]$
\cite{Haiman:2001}.
Let $E_n$ denote the tautological bundle on $\Hilb^n$ and $\psi^k$ the $k$-th
Adams operation.  The weights of $E_n^\vee$ at $I_\lambda$ are
$q^{j-1}t^{1-i}$, for $(i,j)\in\lambda$.

\begin{proposition}\label{prop:kt-tautological}
For every $F\in\Lamsl$ and $n\ge0$ there is a class
$\mathcal K_{F,n}\in\KT(\Hilb^n)_\QQ$, polynomial over $\QQ[q^{\pm},t^{\pm}]$ in the Adams
operations of $E_n^\vee$, such that
\[
\mathcal K_{F,n}|_{I_\lambda}=F(q^\lambda)
\qquad(\lambda\vdash n).
\]
\end{proposition}
\begin{proof}
The box-sum identity used in the proof of Proposition~\ref{prop:mac-skew-laurent}
gives
\[
p^\circ_k(q^\lambda)
=(q^k-1)\sum_{(i,j)\in\lambda}(q^{j-1}t^{1-i})^k
=(q^k-1)\psi^k(E_n^\vee)|_{I_\lambda}.
\]
Write $F=\sum_\kappa a_\kappa P^\circ_\kappa$ with $a_\kappa\in\QQ[q^{\pm},t^{\pm}]$, and set
\[
\mathcal K_{F,n}
=\sum_\kappa a_\kappa\prod_j(q^{\kappa_j}-1)\psi^{\kappa_j}(E_n^\vee).
\]
This finite sum has the required restrictions.  No completion or inversion
of hook factors is needed.
\end{proof}

\begin{proposition}\label{prop:kt-pivots}
For fixed $\mu,\nu$ and $n\ge0$, there is a class
$\mathcal G^{(n)}_{\mu\nu}\in\KT(\Hilb^n)_\QQ$ with restrictions
\[
\mathcal G^{(n)}_{\mu\nu}|_{I_\lambda}
=t^{-n(\lambda)}g^\lambda_{\mu\nu}(q,t)
\qquad(\lambda\vdash n).
\]
For fixed $\lambda,\mu$ and $m\ge0$, there is likewise a class
$\mathcal G^{\lambda,(m)}_{\mu\,\cdot}\in\KT(\Hilb^m)_\QQ$ with restrictions
$t^{-n(\nu)}g^\lambda_{\mu\nu}(q,t)$ for $\nu\vdash m$.
Their edge conditions along pivot curves are the output and input
congruences of Theorem~\ref{thm:mac-pivots-intro}.
\end{proposition}
\begin{proof}
For the output class take
$F=\sh_{q,t}(\Phi^{(n)}_{\mu\nu})\in\Lamsl$.
By Proposition~\ref{prop:mac-top-output-pivots},
$F(q^\lambda)=u_\lambda g^\lambda_{\mu\nu}$ for $\lambda\vdash n$.
The determinant character satisfies
\begin{equation}\label{eq:determinant-twist}
\det(E_n^\vee)|_{I_\lambda}
=q^{n(\lambda')}t^{-n(\lambda)},\qquad
u_\lambda=(-1)^n t^{-n(\lambda)}\det(E_n^\vee)|_{I_\lambda}.
\end{equation}
Thus the genuine class
\[
\mathcal G^{(n)}_{\mu\nu}
=(-1)^n[\det(E_n)]\,\mathcal K_{F,n}
\]
has the stated restrictions.  For the input class use
$F=\sh_{q,t}(\Psi^{(m)}_{\lambda/\mu})$ and
$\mathcal G^{\lambda,(m)}_{\mu\,\cdot}=(-1)^m[\det(E_m)]\mathcal K_{F,m}$.
Proposition~\ref{prop:mac-input-pivots} supplies the evaluation
$F(q^\nu)=u_\nu g^\lambda_{\mu\nu}$.  If $\mu\not\subseteq\lambda$, take the
zero class.

The $K$-theoretic edge condition along a curve of tangent weight $\chi$ is
congruence modulo $1-e^{-\chi}$ \cite[Appendix]{KnutsonRosu:2003}.
For a pivot curve, $\chi=-A\epsilon_1+B\epsilon_2$, so this modulus is
$1-q^At^B=h$.  Applying the edge condition to the classes above gives the
congruences with the factors $t^{-n(\lambda)}$ and $t^{-n(\nu)}$, respectively.
Multiplying by the remaining, fixed powers of $t$ gives the triple
normalization $t^{-n(\mu)-n(\nu)-n(\lambda)}$, and pivots in $\mu$ follow by symmetry.
\end{proof}

The factor $t^{-n(\lambda)}$ by itself is not the restriction of an
equivariant line bundle: across a pivot its ratio is
$t^{-B}\not\equiv1\pmod h$.  What is a line-bundle restriction, up to the
constant sign $(-1)^n$, is the quotient $t^{-n(\lambda)}/u_\lambda$.
Equation~\eqref{eq:determinant-twist} identifies it with
$(-1)^n\det(E_n)|_{I_\lambda}$.  The determinant twist therefore converts the
natural interpolation evaluations into the normalization of the Macdonald
congruences.

\subsection*{Acknowledgements}

The author is grateful to Per Alexandersson and Houcine Ben Dali for helpful conversations. The author utilized GPT 5.5 (High Thinking) and Claude Opus 4.8 as a sounding board on early drafts of this paper, and they provided useful clarifications about the structure of the ring of shifted symmetric functions and Lassalle's shift map. The same models were used to assist with typesetting and editing of this document. Claude Opus 5.5 (Max) supplied the keys ideas for the K-theory section, and the proof of the Macdonald input pivots proof.
\bibliographystyle{utphys}
\bibliography{congruences}

\end{document}